\documentclass[11pt]{amsart}
\usepackage{mysty}
\usepackage[utf8]{inputenc}
\usepackage[OT1]{fontenc}
\usepackage{amsfonts, amsmath, amsthm, amssymb, stmaryrd}
\usepackage{graphicx}
\usepackage{listings}
\usepackage[margin=1in]{geometry}
\usepackage{tikz-cd}
\usepackage{fancyhdr}

\begin{document}

\title{Mod $\lowercase{p}$ Iwasawa algebras of pro-$\lowercase{p}$ Iwahori subgroups}

\author[Ariaz]{Rudy Ariaz}
\address{McGill University Department of Mathematics and Statistics, Burnside Hall, 805 Sherbrooke Street West, Montreal, Quebec H3A 0B9 Canada}
\email{rudy.ariaz@mail.mcgill.ca}

\author[Creech]{Steven Creech}
\address{Rutgers University Department of Mathematics, 110 Frelinghuysen Rd, Piscataway, NJ 08854-8019 USA}
\email{steven.creech@rutgers.edu}

\author[Hu]{Bryan Hu}
\address{UC San Diego Department of Mathematics, 9500 Gilman Drive, La Jolla, CA 92093-0112 USA}
\email{brhu@ucsd.edu}

\author[Khunger]{Simran Khunger}
\address{University of Michigan Department of Mathematics, 2074 East Hall, 530 Church Street, Ann Arbor, MI 48109-1043 USA}
\email{skhunger@umich.edu}

\author[Kozio\l]{Karol Kozio\l}
\address{Department of Mathematics, CUNY Baruch College, 55 Lexington Ave, New York, NY 10010 USA} 
\email{karol.koziol@baruch.cuny.edu}

\author[Rankothge]{Bharatha Rankothge}
\address{UC San Diego Department of Mathematics, 9500 Gilman Drive, La Jolla, CA 92093-0112 USA}
\email{brankoth@ucsd.edu}

\author[Zhang]{Bobby Zixuan Zhang}
\address{Duke University Department of Mathematics, 120 Science Drive, 117 Physics Building, Durham, NC 27708-0320 USA}
\email{bobby.zhang.math@duke.edu}

\date{}


\begin{abstract}
    Suppose $F$ is a finite unramified extension of $\mathbb{Q}_p$, and $G$ is the group of $F$-points of a split, connected, reductive group over $F$.  Under a natural restriction on $p$, we determine the structure of the graded mod $p$ Iwasawa algebra $\textrm{gr}_{\mathfrak{m}}(\mathbb{F}_p [\![ I]\!])$, where $I$ is a pro-$p$ Iwahori subgroup of $G$.  We also determine its maximal commutative quotient, and relate these results to Gelfand--Kirillov dimensions of smooth mod $p$ representations of $G$.

    \bigskip

    Soit $F$ une extension finie non ramifiée de $\mathbb{Q}_p$, et $G$ le groupe des $F$-points d'un groupe réductif connexe déployé sur $F$. Moyennant une hypothèse naturelle sur $p$, on détermine la structure de l'algèbre graduée d'Iwasawa $\textrm{gr}_{\mathfrak{m}}(\mathbb{F}_p [\![ I]\!])$, avec $I$ un sous-groupe pro-$p$ Iwahori de $G$. On détermine également son quotient commutatif maximal, et on relie ces résultats aux dimensions de Gelfand--Kirillov des représentations lisses de $G$ en caractéristique $p$.
\end{abstract}

\maketitle 
\tableofcontents 

\section{Introduction}

Suppose $p$ is a prime number and $F$ is a finite extension of $\bbQ_p$.  The mod $p$ Local Langlands program aims to develop a precise connection between mod $p$ representations (i.e., with $\bbF_p$ coefficients) $\rho$ of $\textnormal{Gal}(\overline{F}/F)$, the absolute Galois group of $F$, and mod $p$ representations $\pi$ of reductive groups over $F$.  To date, such a proposal has been carried out most thoroughly for the group $\textnormal{GL}_2$ over $\bbQ_p$: in this setting, there exists a correspondence between two-dimensional mod $p$ representations of $\textnormal{Gal}(\overline{\bbQ}_p/\bbQ_p)$, and certain smooth mod $p$ representations of $\textnormal{GL}_2(\bbQ_p)$, thanks to the work of Breuil, Berger, Colmez, Kisin, Pa\v{s}k\={u}nas and others (see \cite{berger:bourbaki}, \cite{breuil:bourbaki}, \cite{breuil:ICM} for an overview).

Unfortunately, the situation grows infinitely more complicated beyond the group $\textnormal{GL}_2(\bbQ_p)$.  To give just one illustration, work of Breuil--Pa\v{s}k\={u}nas \cite{BP} shows that for a proper unramified extension $F$ of $\bbQ_p$ and a ``generic'' irreducible representation $\rho$ of $\textnormal{Gal}(\overline{F}/F)$, there is an infinite family of smooth representations of $\textnormal{GL}_2(F)$ associated to $\rho$.  (For further examples of pathologies that occur, see \cite{schraen:finpres}, \cite{wu:finpres}, \cite{GLS}.)  Thus, the $\textnormal{GL}_2(F)$ side of a hypothetical correspondence is much more unruly, and there is no hope for a naive bijection as in the $F = \bbQ_p$ case.

Despite these obstacles, the framework of local-global compatibility has proven to be a fruitful approach for attacking these questions.  Slightly more precisely, we fix a generic Galois representation $\rho$ of $\textnormal{Gal}(\overline{F}/F)$ with $F$ unramified over $\bbQ_p$, and consider the first mod $p$ cohomology of a tower of Shimura curves.  Such a tower has commuting actions of $\textnormal{Gal}(\overline{F}/F)$ and $\textnormal{GL}_2(F)$, and therefore the $r$-isotypic component gives a smooth representation $\pi$ of $\textnormal{GL}_2(F)$.  (Here $r$ is an appropriate globalization of $\rho$.)  The representation $\pi$ is contained in the infinite family constructed by Breuil--Pa\v{s}k\={u}nas (at least if the level is chosen ``minimally;'' see \cite{EGS}), and therefore is a viable candidate to participate in a Local Langlands correspondence.  Moreover, many structural properties of the representation $\pi$ have now been proven: the invariants under a pro-$p$ Iwahori subgroup (\cite{EGS}), the $\textnormal{GL}_2(k_F)$-representation given by $\pi^{1 + p\textnormal{M}_2(\cO_F)}$ (\cite{LMS}, \cite{HuWang},  \cite{Le:wild}), the ``diagram'' associated to $\pi$ (\cite{DottoLe}), and, most importantly for the present article, the \textbf{Gelfand--Kirillov dimension} of $\pi$ (\cite{BHHMS}, \cite{HuWang2}).

In order to obtain bounds on the Gelfand--Kirillov dimension of $\pi$ (a measure of size for infinite-dimensional representations), the authors of \cite{BHHMS} utilize the structure of \textbf{Iwasawa algebras} associated to pro-$p$ Iwahori subgroups.  The purpose of the present article is to extend their analysis to arbitrary split reductive groups over unramified extensions.  More precisely, we assume now that $F$ is unramified over $\bbQ_p$, and $G$ is the group of $F$-points of a split connected reductive group $\bG$.  We assume for simplicity that the root system of $\bG$ is irreducible, and that $p - 1$ is larger than the Coxeter number of $\bG$.  We fix a Borel subgroup $\bB = \bT\bU$ in $\bG$, and let $I$ denote the associated pro-$p$ Iwahori subgroup of $G$.  (In the case of $\textnormal{GL}_2(\bbQ_p)$, one can take $I$ to be the subgroup of matrices with $\bbZ_p$ entries which are upper-triangular and unipotent modulo $p$.)  The Iwasawa algebra of $I$ is defined as
$$\bbF_p \llbracket I \rrbracket := \varprojlim_{\cN \vartriangleleft_o I}\bbF_p[I/\cN],$$
where $\cN$ ranges over the open normal subgroups of $I$.  This is a (in general non-commutative) profinite local ring.  Despite the straightforward definition, working with the above presentation is a bit unwieldy.  In order to gain more insight into the algebra $\bbF_p \llbracket I \rrbracket$, we equip $I$ with a \textbf{$p$-valuation} $\omega$ in the sense of Lazard \cite{lazard}.  This $p$-valuation  induces a filtration on $I$, and the associated graded group $\gr(I)$ has the structure of a graded Lie algebra over a one-variable polynomial algebra $\bbF_p[P]$.  Our first main result is the following (see the body of the article for precise statements):

\begin{thm}[Proposition \ref{prop:overline-gr(I)-pres}]
Let $\Phi = \Phi^+ \sqcup \Phi^-$ denote the root system of $\bG$ associated to a split maximal torus $\bT$, and for each $\alpha \in \Phi$, let $u_\alpha:F \longrightarrow G$ denote the associated root morphism.  Fix also a generator $\xi$ for $k_F$ over $\bbF_p$ and a multiplicative lift $[\xi] \in \mathcal{O}_F$.  Then the Lie algebra $\gr(I)\otimes_{\bbF_p[P]}\bbF_p$ has a basis consisting of (images of) elements of the form:
\begin{enumerate}
\item[(a)] $u_\alpha([\xi]^r)$ for $\alpha \in \Phi^+$ and $0 \leq r \leq [k_F:\bbF_p] - 1$,
\item[(b)] $u_\beta(p[\xi]^r)$ for $\beta \in \Phi^-$ and $0 \leq r \leq [k_F:\bbF_p] - 1$,
\item[(c)] $\lambda(1 + p[\xi]^r)$ for $\lambda$ ranging over a basis of $X_*(\bT)$ and $0 \leq r \leq [k_F:\bbF_p] - 1$.
\end{enumerate}
We have an explicit description of the Lie bracket in terms of this basis.  In particular, the elements in point (c) are central, the Lie subalgebra generated by the elements in point (a) is isomorphic to the Lie algebra of $\textnormal{Res}_{k_F/\bbF_p}(\bU_{k_F})$, while the Lie subalgebra generated by the elements in point (b) is abelian.  
\end{thm}

The $p$-valuation $\omega$ also induces a filtration on the Iwasawa algebra $\bbF_p\llbracket I\rrbracket$, and results of Lazard then show that we have an isomorphism of associative algebras
\begin{equation}
\label{eqn:intro}
U_{\bbF_p}(\gr(I)\otimes_{\bbF_p[P]}\bbF_p)  \stackrel{\sim}{\longrightarrow} \gr_{\omega}(\bbF_p\llbracket I \rrbracket),
\end{equation}
where the right-hand side denotes the associated graded algebra.  On the other hand, denoting by $\fm$ the unique maximal ideal of $\bbF_p\llbracket I\rrbracket$, we can also consider $\gr_{\fm}(\bbF_p\llbracket I\rrbracket)$, the associated graded algebra for the $\fm$-adic filtration.  Our second main result relates these two filtrations:

\begin{thm}[Proposition \ref{prop: ideal-comparison}]
\label{thm:intro2}
Suppose $\bG$ is semisimple.  Then, after rescaling indices, we have an isomorphism of graded algebras
$$\gr_{\omega}(\bbF_p\llbracket I\rrbracket) \cong \gr_{\fm}(\bbF_p\llbracket I\rrbracket).$$
\end{thm}

\noindent In fact, this result also holds in the reductive case, assuming we take quotients by the center of $I$ on both sides.

Combining the above two theorems with equation \eqref{eqn:intro}, we obtain an explicit presentation of the graded algebra $\gr_{\fm}(\bbF_p\llbracket I\rrbracket)$.  This allows us to compute a minimal generating set for $\gr_{\fm}(\bbF_p\llbracket I\rrbracket)$ (Corollary \ref{prop: Minimal Generating Set For Universal Enveloping Algebra}), as well as its maximal commutative quotient (Corollary \ref{prop:largest-comm-quot}).

In the final section, we return to our goal of understanding Gelfand--Kirillov dimension.  We obtain in Corollary \ref{cor:GKcor} a result which stands in contrast to that of \cite{BHHMS}: if $\pi$ is a smooth mod $p$ representation of $G$ with central character for which the action of $\gr_{\fm}(\bbF_p\llbracket I\rrbracket)$ on $\gr_{\fm}(\pi^\vee)$ factors through the largest commutative quotient of $\gr_{\fm}(\bbF_p\llbracket I\rrbracket)$, then the bound on Gelfand--Kirillov dimension we obtain is ``too small'' for most groups $\bG$ (namely, it contradicts global expectations coming from completed cohomology).  Therefore our results suggest that for those $\pi$ which come from cohomology of locally symmetric spaces, we will necessarily have to work with non-commutative quotients of $\gr_{\fm}(\bbF_p\llbracket I\rrbracket)$.

\subsection*{Related work}
While this article was being prepared, Dotto--Le Hung posted the preprint \cite{DottoLeHung} on the arXiv.  Therein, they obtain results similar to ours on the structure of $\gr(I)\otimes_{\bbF_p[P]}\bbF_p$, with similar restrictions ($\bG$ split, $F/\bbQ_p$ unramified, $p - 1$ larger than the Coxeter number of $\bG$).  We would like to point out that we were not aware of their ongoing work, nor they of ours.  Though there is some overlap in the methods used, there are also several differences.  On the one hand, Dotto--Le Hung obtain a slightly more refined structural result (writing $\gr(I)\otimes_{\bbF_p[P]}\bbF_p$ as a semidirect product), and ultimately use it to compute the cohomology groups $H^*(\bG(\cO_F),\bbF_p)$.  On the other hand, our focus was on Gelfand--Kirillov dimension, and the comparison of Theorem \ref{thm:intro2} does not appear in \cite{DottoLeHung}.

\subsection*{Acknowledgements}
This project began at the Rethinking Number Theory workshop during the summer of 2024, and we would like thank the organizers (Allechar Serrano L\'opez, Heidi Goodson, and Mckenzie West) for the opportunity to participate.  We would also like to thank Nischay Reddy for his involvement in the project during the workshop.  We thank the anonymous referee for his/her insightful suggestions, which improved some of our original proofs.

During the preparation of this article, we received support from the following sources: RA was supported by the Fonds de recherche du Québec (FRQ), grants 344463 and 364034; BH was supported by NSF grant 2144021; SK was supported by NSF grant DGE-2241144; KK was supported by NSF grant DMS-2310225 and a PSC-CUNY Trad B award.

\section{Notation and preliminaries}

We begin by collecting the basic notations and results we will need.

\subsection{\texorpdfstring{$p$-valuations}{p-valuations}}
\label{subsec:p-valuations}

We first recall some constructions of Lazard \cite[\S\S II.1, III.2]{lazard}.

Suppose $\cG$ is a compact $p$-adic analytic group (in particular, $\cG$ is profinite).  A function $\omega:\cG \longrightarrow \bbR_{ > 0} \cup \{\infty\}$ is called a \textbf{$p$-valuation} if it satisfies the following conditions for all $g,h \in \cG$ (cf. \cite[Defs. II.1.1.1, II.1.2.10, III.2.1.2]{lazard}):
\begin{itemize}
\item $\omega(g) > \frac{1}{p - 1}$,
\item $\omega(g) = \infty$ if and only if $g = 1$,
\item $\omega(h^{-1}g) \geq \min\{\omega(g), \omega(h)\}$,
\item $\omega(ghg^{-1}h^{-1}) \geq \omega(g) + \omega(h)$,
\item $\omega(g^p) = \omega(g) + 1$.
\end{itemize}
By \cite[Prop. III.2.2.6]{lazard}, $\omega(\cG\smallsetminus\{1\})$ is a discrete subset of $\bbR_{>0}$.

Given a real number $\nu\in \bbR_{> 0}$, we define 
$$\cG_\nu := \{g\in \cG: \omega(g) \geq \nu\} \qquad\textnormal{and} \qquad \cG_{\nu^+} := \{g\in \cG: \omega(g) > \nu\}$$
(\cite[\S II.1.1.2]{lazard}).  The groups $\cG_\nu, \cG_{\nu^+}$ are normal in $\cG$, and $\cG_\nu/\cG_{\nu^+}$ is central in $\cG/\cG_{\nu^+}$.  We then define
$$\gr(\cG) := \bigoplus_{\nu \in \bbR_{>0}}\gr_\nu(\cG) := \bigoplus_{\nu \in \bbR_{>0}}\cG_\nu/\cG_{\nu^+}.$$
Given an element $x \in \cG \smallsetminus \{1\}$, we let $\gr(x):= x\cG_{\omega(x)^+}$ denote its image in $\gr_{\omega(x)}(\cG)$.  According to \cite[\S II.1.1.7, III.2.1.2]{lazard}, the group $\gr(\cG)$ has the structure of a graded Lie algebra over $\bbF_p$, where the Lie bracket is induced from the commutator in $\cG$: if $g\in \cG_\nu, h \in \cG_\mu$, we have
$$\left[g\cG_{\nu^+}, h\cG_{\mu^+}\right]:= ghg^{-1}h^{-1}\cG_{(\nu + \mu)^+}.$$

The Lie algebra $\gr(\cG)$ has an additional piece of structure.  Given $g\in \cG_\nu$, we define 
\begin{eqnarray*}
	P: \gr_\nu(\cG) & \longrightarrow & \gr_{\nu + 1}(\cG) \\
	g\cG_{\nu^+} & \longmapsto & g^p\cG_{(\nu + 1)^+}.
\end{eqnarray*}
This map is $\bbF_p$-linear, and assembles to give a linear map $P:\gr(\cG) \longrightarrow \gr(\cG)$ which is homogeneous of degree 1.  Furthermore, the Lie bracket is bilinear for the operator $P$, and therefore $\gr(\cG)$ is a Lie algebra over the ring $\bbF_p[P]$.  (See \cite[\S. III.2.1]{lazard}.)  In the setting where $\cG$ is a compact $p$-adic analytic group, $\gr(\cG)$ is \textit{free} over $\bbF_p[P]$ of rank equal to the dimension of $\cG$ \cite[Prop. III.3.1.3]{lazard}.

\subsection{Iwasawa algebras}

Suppose again that $\cG$ is a compact $p$-adic analytic group, and let us write $\cG = \varprojlim_{\cN \vartriangleleft_o \cG} \cG/\cN$, where $\cN$ ranges over the open normal subgroups of $\cG$.  We define the \textbf{completed group ring} or \textbf{Iwasawa algebra} of $\cG$ over $\bbZ_p$ as
$$\bbZ_p \llbracket \cG \rrbracket := \varprojlim_{\cN \vartriangleleft_o \cG}\bbZ_p[\cG/\cN].$$
We also define its mod $p$ version, $\bbF_p [\![ \cG]\!]$, as the mod $p$ reduction of $\bbZ_p [\![ \cG]\!]$, or, equivalently, as
$$\bbF_p \llbracket \cG \rrbracket := \varprojlim_{\cN \vartriangleleft_o \cG}\bbF_p[\cG/\cN].$$
Under the additional assumption that $\cG$ is pro-$p$, $\bbF_p\llbracket \cG \rrbracket$ is a complete Noetherian local ring whose maximal ideal we denote by $\fm$.  We then denote by $\gr_{\fm}(\bbF_p\llbracket \cG \rrbracket)$ the associated graded ring for the $\fm$-adic filtration:
$$\gr_{\fm}(\bbF_p\llbracket \cG \rrbracket) :=  \bigoplus_{n = 0}^{\infty} \fm^n/\fm^{n + 1}$$

Suppose now $\cG$ is equipped with a $p$-valuation $\omega$.  We let $w: \bbZ_p\llbracket \cG \rrbracket \longrightarrow \bbR_{>0} \cup \{\infty\}$ be the valuation associated to $\omega$ defined in \cite[Def. III.2.3.1.2]{lazard} (note that the ``filtration'' $w$ on $\bbZ_p[\cG]$ defined in \emph{loc. cit.} is in fact a valuation (as defined in \cite[Def. I.2.2.1]{lazard}), and $\bbZ_p\llbracket \cG \rrbracket$ is the completion of $\bbZ_p[\cG]$ for $w$ \cite[Thm. III.2.3.3, Cor. III.2.3.4]{lazard}).  
We let $\textnormal{gr}_{\omega}(\bbZ_p\llbracket \cG \rrbracket)$ denote the graded ring associated to $w$, defined as 
$$\textnormal{gr}_{\omega}(\bbZ_p\llbracket \cG \rrbracket) := \bigoplus_{\nu \in \bbR_{\geq 0}} \textnormal{gr}_\nu(\bbZ_p\llbracket \cG \rrbracket) := \bigoplus_{\nu \in \bbR_{\geq 0}}\bbZ_p\llbracket \cG \rrbracket_{\nu}/\bbZ_p\llbracket \cG \rrbracket_{\nu^+},$$
where 
$$\bbZ_p\llbracket \cG \rrbracket_\nu := \{\tau \in \bbZ_p\llbracket \cG \rrbracket: w(\tau) \geq \nu\}\qquad\textnormal{and}\qquad \bbZ_p\llbracket \cG \rrbracket_{\nu^+} := \{\tau \in\bbZ_p\llbracket \cG \rrbracket: w(\tau) > \nu\}.$$

Analogously to $\gr(\cG)$, the graded ring $\textnormal{gr}_{\omega}(\bbZ_p\llbracket \cG \rrbracket)$ has the structure of a graded algebra over $\bbF_p[P]$, where $P$ acts by 
\begin{eqnarray*}
P: \textnormal{gr}_\nu(\bbZ_p\llbracket \cG \rrbracket) & \longrightarrow & \textnormal{gr}_{\nu + 1}(\bbZ_p\llbracket \cG \rrbracket) \\
\tau + \bbZ_p\llbracket \cG \rrbracket_{\nu^+ } & \longmapsto & p\tau + \bbZ_p\llbracket \cG \rrbracket_{(\nu + 1)^+}.
\end{eqnarray*}
The maps (defined in \cite[II.1.1.9]{lazard})
\begin{eqnarray*}
\textnormal{gr}_\nu(\cG) & \longrightarrow & \textnormal{gr}_\nu(\bbZ_p\llbracket \cG \rrbracket)  \\
g\cG_{\nu^+} & \longmapsto & (g - 1) + \bbZ_p\llbracket \cG \rrbracket_{\nu^+} 
\end{eqnarray*}
assemble to a map $\gr(\cG) \longrightarrow \gr_{\omega}(\bbZ_p\llbracket \cG \rrbracket)$.

We have the following result on the structure of $\gr_{\omega}(\bbZ_p\llbracket \cG \rrbracket)$ (see \cite[Thm. III.2.3.3]{lazard}):

\begin{thm}
The maps above induce an isomorphism of graded $\bbF_p[P]$-modules
$$U_{\bbF_p[P]}(\gr(\cG)) \stackrel{\sim}{\longrightarrow} \gr_{\omega}(\bbZ_p\llbracket \cG \rrbracket),$$
where $U_{\bbF_p[P]}$ denotes the universal enveloping algebra over $\bbF_p[P]$.
\end{thm}

We define $\bbF_p\llbracket \cG \rrbracket_{\nu}$ (resp., $\bbF_p\llbracket \cG \rrbracket_{\nu^+}$) as the image of $\bbZ_p\llbracket \cG \rrbracket_{\nu}$ (resp., $\bbZ_p\llbracket \cG \rrbracket_{\nu^+}$) in $\bbF_p\llbracket \cG \rrbracket$, and define $\gr_{\omega}(\bbF_p\llbracket \cG\rrbracket)$ as the associated graded ring. (When $\cG$ is equal to a pro-$p$ Iwahori subgroup, we will compare the graded rings $\gr_{\fm}(\bbF_p\llbracket \cG\rrbracket)$ and $\gr_{\omega}(\bbF_p\llbracket \cG\rrbracket)$ in Proposition \ref{prop: ideal-comparison}.)  We note that $\gr_\omega(\bbF_p\llbracket \cG\rrbracket)$ may also be defined using the quotient filtration $\overline{w}$ \cite[I.2.1.7, I.2.3.8]{lazard}.  

Sending $P$  to  $0$ in the previous result gives the following.

\begin{cor}\label{lazardcor}
We have an isomorphism of graded $\bbF_p$-algebras
$$U_{\bbF_p}(\gr(\cG)\otimes_{\bbF_p[P]}\bbF_p) = U_{\bbF_p[P]}(\gr(\cG))\otimes_{\bbF_p[P]}\bbF_p \stackrel{\sim}{\longrightarrow} \gr_{\omega}(\bbF_p\llbracket \cG \rrbracket).$$
\end{cor}

Finally, if $\bbF$ is a finite extension of $\bbF_p$, we define analogous objects and constructions (e.g., $\bbF\llbracket \cG\rrbracket_{\nu^+}, \gr_{\fm}(\bbF\llbracket \cG\rrbracket)$, etc.) by applying the base change $-\otimes_{\bbF_p}\bbF$.

\subsection{Reductive groups}
Let $F$ denote a finite, unramifed extension of $\bbQ_p$ of degree $f \geq 1$, with ring of integers $\cO_F$, maximal ideal $\fm_F$, and residue field $k_F$ of size $q := p^f$.  We let $\textnormal{val}_p:F \longrightarrow \bbZ \cup \{\infty\}$ denote the valuation on $F$ normalized so that $\textnormal{val}_p(p) = 1$.

Fix a connected reductive group $\mathbf{G}$, defined and split over $F$.  We also use the notation $\mathbf{G}$ to denote a split connected reductive integral model over $\cO_F$, which exists by \cite[Thm. 1.2]{conrad:nonsplit}.  Let $\mathbf{Z}$ denote the connected center of $\mathbf{G}$.  We fix a Borel subgroup $\mathbf{B} = \mathbf{T}\mathbf{U}$, where $\mathbf{T}$ is an $F$-split maximal torus of $\mathbf{G}$ and $\mathbf{U}$ is the unipotent radical of $\mathbf{B}$.  We use the same letters to denote the $\cO_F$-models of all these subgroups in the $\cO_F$-group scheme $\mathbf{G}$.  Finally, we let $G := \mathbf{G}(F)$ denote the group of $F$-rational points of $\mathbf{G}$ (with analogous notation for other $\cO_F$-groups).

Let $X^*(\mathbf{T})$ and $X_*(\mathbf{T})$ denote the character and cocharacter lattices of $\mathbf{T}$, respectively.  Let $\Phi \subset X^*(\mathbf{T})$ denote the subset of roots of $\mathbf{G}$ relative to $\mathbf{T}$.  The choice of $\mathbf{B}$ defines a subset of positive roots $\Phi^+$; we let $\Delta \subset \Phi^+$ denote the associated set of simple roots and $\Phi^- := -\Phi^+$ the set of negative roots, so that $\Phi = \Phi^+ \sqcup \Phi^-$.  Further, we let $\textnormal{ht}:\Phi \longrightarrow \bbZ$ denote the height function relative to $\Delta$.  We assume that $\Phi$ is irreducible, and let $h$ denote the Coxeter number of $\mathbf{G}$, defined as $1 + \max_{\alpha \in \Phi}\{\textnormal{ht}(\alpha)\}$.

We fix once and for all a total order on $\Phi$ refining the partial order induced by $\textnormal{ht}$.  Given $\alpha \in \Phi^+$, we let 
$$\varphi_\alpha:\mathbf{SL}_{2/\cO_F} \longrightarrow \mathbf{G}$$ 
denote the morphism of $\cO_F$-group schemes defined as in \cite[Exer. 6.5.1]{conrad:reductive}.  The coroot $\alpha^\vee :\mathbf{G}_{m/\cO_F} \longrightarrow \mathbf{T}$ is then given as
$$\alpha^\vee(x) = \varphi_\alpha\left(\begin{pmatrix}
    x & 0 \\ 0 & x^{-1}
\end{pmatrix}\right).$$
Further, we define $u_\alpha, u_{-\alpha}:\mathbf{G}_{a/\cO_F} \longrightarrow \mathbf{G}$ as
$$u_\alpha(x) := \varphi_\alpha\left(\begin{pmatrix} 1 & x \\ 0 & 1\end{pmatrix}\right), \qquad u_{-\alpha}(x) := \varphi_\alpha\left(\begin{pmatrix} 1 & 0 \\ x & 1\end{pmatrix}\right).$$
For any $\alpha \in \Phi$, the image of $u_\alpha$ is the root subgroup $\mathbf{U}_\alpha$, and we have $tu_{\alpha}(x)t^{-1} = u_\alpha(\alpha(t)x)$ for $t\in T, x \in F$.  Moreover, by \cite[Prop. 5.1.16]{conrad:reductive}, multiplication induces isomorphisms of $\cO_F$-schemes
\begin{equation}
\label{unip-mult}
  \prod_{\alpha \in \Phi^+}\mathbf{U}_\alpha \stackrel{\sim}{\longrightarrow}\mathbf{U},\qquad \prod_{\alpha \in \Phi^-}\mathbf{U}_\alpha \stackrel{\sim}{\longrightarrow}\mathbf{U}^-  
\end{equation}
where $\mathbf{U}^-$ denotes the unipotent radical of the Borel subgroup opposite to $\mathbf{B}$, and the products are taken in the fixed total order.

The root morphisms $u_\alpha$ satisfy the following commutator formula \cite[Prop. 5.1.14]{conrad:reductive}: if $\beta \neq \pm \alpha$, then
\begin{equation}
\label{commutator}
    [u_\alpha(x),u_\beta(y)] := u_\alpha(x)u_\beta(y)u_\alpha(x)^{-1}u_\beta(y)^{-1} = \prod_{\substack{i,j > 0 \\ i\alpha + j\beta \in \Phi}}u_{i\alpha + j\beta}(c_{\alpha,\beta;i,j}x^iy^j),
\end{equation}
the product taken in the fixed total order.  Here, the $c_{\alpha,\beta;i,j}$ are fixed \textit{integers} (since $\mathbf{G}$ and all of the morphisms above can be defined over $\mathbb{Z}$) which are determined by the collection of morphisms $\{\varphi_\alpha\}_{\alpha \in \Phi^+}$ (see also \cite[\S\S 3.2.1 -- 3.2.3]{BT2}).  More precisely, we have $c_{\alpha,\beta;i,j} \in \{\pm 1, \pm 2, \pm 3\}$ (\cite[Exp. XXIII, Prop. 6.4]{SGA3}).

\subsection{\texorpdfstring{$p$-valuations on $p$-adic reductive groups}{p-valuations on p-adic reductive groups}}

Our next task will be to define a $p$-valuation on the pro-$p$ Iwahori subgroup of a $p$-adic reductive group.

We begin with a lemma.  In the following, we define $T_n := \textnormal{ker}(\mathbf{T}(\cO_F) \longrightarrow \mathbf{T}(\cO_F/p^n\cO_F))$ for $n \in \mathbb{Z}_{\geq 0}$ (and analogously for $\mathbf{Z}$).  Note that $T_n$ is the subgroup generated by $\lambda(z)$ as $\lambda$ ranges over $X_*(\mathbf{T})$ and $z$ ranges over $1 + p^n\cO_F$.

\begin{lem}
	\label{lem:boundconsq's}
	Suppose $p > h + 1$, where $h$ denotes the Coxeter number of $\mathbf{G}$.  Then:
	\begin{enumerate}
		\item We have $c_{\alpha,\beta;i,j} \in \cO_F^\times \cap \bbZ$ (whenever these numbers are defined).
		\item For $n \geq 1$, we have $T_n \cong Z_n \times \langle \alpha^\vee(1 + p^n\cO_F)\rangle_{\alpha \in \Phi}$ as topological groups.
		\item Every pro-$p$ subgroup of $G$ is torsion-free.
	\end{enumerate}
\end{lem}

\begin{proof}
	\begin{enumerate}
		\item Since $h$ is always at least 2, our assumption implies $p > h + 1 \geq 3$, which gives the claim.
		\item Let $\mathbf{G}^{\textnormal{sc}} \longrightarrow \mathbf{G}/\mathbf{Z}$ denote the simply connected cover of the semisimple group $\mathbf{G}/\mathbf{Z}$ (cf. \cite[Exer. 6.5.2]{conrad:reductive}).  The preimage of $\mathbf{T}/\mathbf{Z}$ is a maximal torus $\mathbf{T}^{\textnormal{sc}}$ in $\mathbf{G}^{\textnormal{sc}}$, and we obtain a short exact sequence
		$$1 \longrightarrow \mathbf{K} \longrightarrow \mathbf{T}^{\textnormal{sc}}\longrightarrow \mathbf{T}/\mathbf{Z} \longrightarrow 1$$
		of algebraic groups, where $\mathbf{K}$ is finite and central in $\mathbf{G}^{\textnormal{sc}}$.  Its order is therefore bounded above by 4 if $\Phi$ is not of type $\mathsf{A}_n$, and bounded above by $n + 1 = h$ if $\Phi$ is of type $\mathsf{A}_n$ (for example, see \cite[\S 2.1.13]{platonovrapinchuk}).  Thus, by taking character groups of the above exact sequence and dualizing, we see that the image of $X_*(\mathbf{T}^{\textnormal{sc}})$ is a finite-index subgroup of $X_*(\mathbf{T})/X_*(\mathbf{Z})$, and moreover that this index is prime to $p$ (using the assumption $p > h + 1$).  This image is precisely $\bigoplus_{\alpha \in \Delta} \bbZ\alpha^\vee$ (i.e., the span of the coroots), from which we conclude that the index $N$ of $X_*(\mathbf{Z}) \oplus \bigoplus_{\alpha \in \Delta}\bbZ\alpha^\vee$ in $X_*(\mathbf{T})$ is prime to $p$.

		Suppose now that $\lambda \in X_*(\mathbf{T})$, and let $z \in 1 + p^n\cO_F$.  By the above remarks, we may write $N\lambda$ as $\lambda_Z + \lambda'$, where $\lambda_Z \in X_*(\mathbf{Z})$ and $\lambda' \in \bigoplus_{\alpha \in \Delta}\bbZ\alpha^\vee$.  Furthermore, since $N$ and $p$ are coprime, the map $x \longmapsto x^N$ is an automorphism of $1 + p^n\cO_F$.  Therefore, we obtain
		$$\lambda(z) = \lambda(z^{1/N})^N = (N\lambda)(z^{1/N}) = (\lambda_Z + \lambda')(z^{1/N}) = \lambda_Z(z^{1/N}) \lambda'(z^{1/N}).$$
		This shows that $T_n$ is generated by $Z_n$ and $\langle\alpha^\vee(1 + p^n\cO_F)\rangle_{\alpha \in \Delta}$.

		It suffices to show that these subgroups have trivial intersection in $T_n$.  Suppose $t = \prod_{\alpha \in \Delta} \alpha^\vee(z_\alpha)$ is an element in this intersection for some $z_\alpha \in 1 + p^n\cO_F$.  Using the fact that $t$ is central, we obtain 
		$$1 = \beta(t) = \prod_{\alpha \in \Delta}\beta(\alpha^\vee(z_\alpha)) = \prod_{\alpha \in \Delta} z_\alpha^{\langle \beta, \alpha^\vee\rangle}$$
		for all $\beta \in \Delta$.  The assumption $p > h + 1$ implies that the determinant of the Cartan matrix of $\Phi$ is prime to $p$ (\cite[Planches I -- IX]{bourbaki:lie}), which shows that the only solution to the above system of equations is the trivial one (for example, by applying the $p$-adic logarithm).  This finishes the claim.  
		\item Suppose $J$ is a pro-$p$ subgroup of $G = \mathbf{G}(F)$, and suppose $j \in J$ satisfies $j^p = 1$.  Consider the short exact sequence
		$$1 \longrightarrow \mathbf{Z} \longrightarrow \mathbf{G} \longrightarrow \mathbf{G}/\mathbf{Z} \longrightarrow 1$$
		of algebraic groups.  Since $\mathbf{Z}$ is a split torus, by Hilbert's Theorem 90 we have $H^1(F,\mathbf{Z}) = 1$.  Thus, taking $F$-points of the above sequence is exact, and we obtain a short exact sequence of topological groups
		$$1 \longrightarrow Z \longrightarrow G \longrightarrow (\mathbf{G}/\mathbf{Z})(F) \longrightarrow 1.$$
		Since $\mathbf{G}/\mathbf{Z}$ is semisimple, the assumption $p > h + 1$ and \cite[Prop. 12.1]{totaro} imply that the image of $J$ in $(\mathbf{G}/\mathbf{Z})(F)$ is torsion-free.  In particular, the image of $j$ in $(\mathbf{G}/\mathbf{Z})(F)$ is trivial, so $j \in J \cap Z$.  The maximal pro-$p$ subgroup of $Z$ is isomorphic to $(1 + p\cO_F)^r$ for some $r$, and this group has no $p$-torsion by our assumption that $F$ is unramified over $\bbQ_p$ (and $p > 2$).  Thus, $j = 1$, and we are done.  
	\end{enumerate}
\end{proof}

\textbf{We assume henceforth that $p > h + 1$.}  We let $I \subset \mathbf{G}(\cO_F)$ denote the pro-$p$ Iwahori subgroup associated to $\Phi^+$.  Using \cite[Prop. I.2.2]{schneiderstuhler} and the morphisms \eqref{unip-mult}, we obtain the Iwahori factorization: multiplication induces a homeomorphism
\begin{equation}
\label{iwahori-prod}
    \left(\prod_{\beta \in \Phi^-}u_{\beta}(p\cO_F)\right)\times T_1\times \left(\prod_{\alpha \in \Phi^+}u_\alpha(\cO_F)\right) \stackrel{\sim}{\longrightarrow} I.
\end{equation}
The $p$-valuation on $I$ will be obtained using this factorization.  We first define three functions:
\begin{itemize}
\item Define a function $\omega_+: I \cap U \longrightarrow \mathbb{R}_{>0} \cup \{\infty\}$ as follows.  Given $i_+ \in I \cap U$, let us write $i_+ = \prod_{\alpha \in \Phi^+}u_\alpha(x_\alpha)$ with $x_\alpha \in \cO_F$.  We then set 
$$\omega_+(i_+) := \min\left\{\textnormal{val}_p(x_\alpha) + \frac{\textnormal{ht}(\alpha)}{h} ~:~ \alpha\in \Phi^+\right\}.$$
\item Define a function $\omega_-: I \cap U^- \longrightarrow \mathbb{R}_{>0} \cup \{\infty\}$ as follows.  Given $i_- \in I \cap U^-$, let us write $i_- = \prod_{\beta \in \Phi^-}u_\beta(x_\beta)$ with $x_\beta \in p\cO_F$.  We then set 
$$\omega_-(i_-) := \min\left\{\textnormal{val}_p(x_\beta) + \frac{\textnormal{ht}(\beta)}{h} ~:~ \beta\in \Phi^-\right\}.$$
\item Finally, let us define a function $\omega_T:T_1 \longrightarrow \mathbb{R}_{>0} \cup \{\infty\}$ by 
$$\omega_T(t) := \sup\{n \in \mathbb{Z}_{\geq 0} ~|~ t \in T_n\}.$$
\end{itemize}
We then define $\omega:I \longrightarrow \mathbb{R}_{>0} \cup \{\infty\}$ as follows: given $i \in I$, we use \eqref{iwahori-prod} to write $i = i_- t i_+$ with $i_- \in I \cap U^-$, $i_+ \in I \cap U$, and $t \in T_1$.  We then set
\begin{equation}
    \label{p-valuation}
    \omega(i) := \min\left\{\omega_-(i_-),~\omega_+(i_+),~\omega_T(t)\right\}.
\end{equation}
According to \cite[Prop. 3.5]{lahirisorensen}, the function $\omega$ defines a $p$-valuation on $I$.

\section{\texorpdfstring{The graded Lie algebra of $I$}{The graded Lie algebra of I}}

Our first task will be to find an explicit presentation for the graded $\mathbb{F}_p$-Lie algebra associated to the $p$-valued group $(I,\omega)$.  We begin by computing the filtered pieces $I_\nu$ and $I_{\nu^+}$, then calculate bases for the graded pieces $\gr_\nu(I) := I_\nu/I_{\nu^+}$ and the $P$ operator $P:\gr_\nu(I)\longrightarrow\gr_{\nu+1}(I)$, and finally compute Lie brackets.

\subsection{\texorpdfstring{The filtration on $I$}{The filtration on I}}

For any $\nu\in\R$, we define
$$I_\nu := \left\{i \in I: \omega(i) \geq \nu\right\},\qquad I_{\nu^+} := \left\{i \in I: \omega(i) > \nu\right\},$$
and set $\gr_\nu(I) := I_\nu/I_{\nu^+}$.  Note that $\omega$ takes values in the set $\frac{1}{h}\bbZ_{\geq 1}$;  therefore, we have $\gr_\nu(I) = 1$ for $\nu \not\in \frac{1}{h}\bbZ_{\geq 1}$, and it suffices to compute the filtered pieces for $\nu\in \frac{1}{h}\bbZ_{\geq 1}$.

For $k \in \{1,2,\ldots, h - 1\}$, define
$$\Phi_k := \{\gamma \in \Phi: \textnormal{ht}(\gamma) \equiv k~(\textnormal{mod}~h)\},$$
the \textbf{roots of height $k$ modulo $h$}.  Similarly, we define 
$$\Phi_k^- := \Phi_k\cap \Phi^-, \qquad \Phi_k^+ := \Phi_k\cap\Phi^+,$$ 
and note that $\gamma\in\Phi_k$ if and only if $-\gamma\in\Phi_{h-k}$.  This implies that any $\beta \in \Phi^-$ satisfies $\beta \in \Phi_{h + \textnormal{ht}(\beta)}$.

According to the homeomorphism \eqref{iwahori-prod}, we write an element $i \in I$ as 
\[
i = i_- t i_+ = \prod_{k=1}^{h-1}\prod_{\beta\in\Phi_{k}^-}u_{\beta}(px_{\beta})\cdot  t \cdot \prod_{k=1}^{h-1}\prod_{\alpha\in\Phi_k^+}u_{\alpha}(x_{\alpha}),
\]
where $x_\beta,x_\alpha\in \cO_F$, and $t\in T_1$.  If $t \in T_n \smallsetminus T_{n + 1}$, then by equation \eqref{p-valuation} we get
\begin{align}
\omega(i)&=\min\left\{n,~ \textnormal{val}_p(px_\beta)+\frac{\textnormal{ht}(\beta)}{h},~\textnormal{val}_p(x_\alpha)+\frac{\textnormal{ht}(\alpha)}{h} ~:~ \beta\in\Phi^-,\alpha\in\Phi^+\right\} \notag\\
&=\min\left\{n,~ \textnormal{val}_p(x_\beta)+\left(1+\frac{\textnormal{ht}(\beta)}{h}\right),~\textnormal{val}_p(x_\alpha)+\frac{\textnormal{ht}(\alpha)}{h} ~:~ \beta\in\Phi^-,\alpha\in\Phi^+\right\} \notag\\
&=\min\left\{n,~\textnormal{val}_p(x_{\beta})+\frac{k}{h},~\textnormal{val}_p(x_{\alpha})+\frac{k}{h} ~:~ \beta\in\Phi^-_{k},\alpha\in\Phi_k^+,k = 1,2,\ldots, h - 1\right\} \notag\\
&=\min\left\{n,~\textnormal{val}_p(x_{\gamma})+\frac{k}{h} ~:~\gamma\in\Phi_k, k = 1,2,\ldots, h - 1\right\}. \label{eqn:omega-hgt}
\end{align}

Now that we have chosen coordinates well-adapted to the $p$-valuation $\omega$, we can calculate the filtered pieces $I_\nu$ and $I_{\nu^+}$. 
\begin{itemize}
\item We begin with $I_n$ for $n\in\Z_{\geq 1}$. By equation \eqref{eqn:omega-hgt}, for an element $i=i_-ti_+$ to lie in $I_n$, we must have $t\in T_n$ and $\textnormal{val}_p(x_{\gamma})\geq n$ for all $\gamma\in \Phi$. Thus, we have
\begin{equation}
\label{eqn:gr-integral1}
I_n = \left(\prod_{\beta \in \Phi^-}u_{\beta}(p^{n+1}\cO_F)\right)\cdot T_n\cdot\left(\prod_{\alpha \in \Phi^+}u_\alpha(p^n\cO_F)\right).    
\end{equation}
To obtain $I_{n^+}$, we now must require that $t\in T_{n+1}$ and $\textnormal{val}_p(x_{\gamma})\geq n$. Thus, we get
\begin{equation}
\label{eqn:gr-integral2}
I_{n^+} = \left(\prod_{\beta \in \Phi^-}u_{\beta}(p^{n+1}\cO_F)\right)\cdot T_{n+1}\cdot \left(\prod_{\alpha \in \Phi^+}u_\alpha(p^n\cO_F)\right).
\end{equation}

\item Next, we consider non-integral jumps, so fix $n\in \Z_{\geq 0}$ and $k\in\{1,\ldots,h-1\}$.  Using equation \eqref{eqn:omega-hgt}, we see for an element $i = i_-ti_+$ to lie in $I_{n+\frac{k}{h}}$, we must have $t\in T_{n+1}$, along with the following condition: for $i < k$, we have $x_{\gamma}\in p^{n+1}\cO_F$ for all $\gamma \in \Phi_i$, and for $i\geq k$, we have $x_{\gamma}\in p^n\cO_F$ for all $\gamma \in \Phi_i$. Consequently, the group $I_{n + \frac{k}{h}}$ takes the form
\begin{align}
I_{n+\frac{k}{h}} &= \left(\prod_{i=1}^{k - 1}\prod_{\beta\in\Phi_{i}^-}u_{\beta}(p^{n+2}\cO_F)\right)\cdot\left(\prod_{i=k}^{h-1}\prod_{\beta\in\Phi_{i}^-} u_{\beta}(p^{n+1}\cO_F)\right)\cdot T_{n+1} \notag\\
& \qquad\cdot \left(\prod_{i=1}^{k-1}\prod_{\alpha\in\Phi_i^+}u_{\alpha}(p^{n+1}\cO_F)\right)\cdot\left(\prod_{i=k}^{h-1}\prod_{\alpha\in\Phi_i^+}u_{\alpha}(p^n\cO_F)\right). \label{eqn:gr-nonintegral1}
\end{align}
To calculate $I_{(n+\frac{k}{h})^+}$, we see that the change occurs only for the index $i =k$: instead of $x_{\gamma}\in p^n\cO_F$ for all $\gamma \in \Phi_k$, we impose the stronger condition $x_{\gamma}\in p^{n+1}\cO_F$. Thus we see that
\begin{align}
I_{(n+\frac{k}{h})^+} &= \left(\prod_{i=1}^{k}\prod_{\beta\in\Phi_{i}^-}u_{\beta}(p^{n+2}\cO_F)\right)\cdot\left(\prod_{i=k + 1}^{h-1}\prod_{\beta\in\Phi_{i}^-} u_{\beta}(p^{n+1}\cO_F)\right)\cdot T_{n+1} \notag\\
& \qquad \cdot \left(\prod_{i=1}^{k}\prod_{\alpha\in\Phi_i^+}u_{\alpha}(p^{n+1}\cO_F)\right)\cdot\left(\prod_{i=k+1}^{h-1}\prod_{\alpha\in\Phi_i^+}u_{\alpha}(p^n\cO_F)\right). \label{eqn:gr-nonintegral2}
\end{align}
\end{itemize}

\subsection{\texorpdfstring{The graded pieces $\gr_\nu(I)$}{The graded pieces gr_nu(I)}} \label{sec: graded-pieces-structure}

We now calculate generators for the graded pieces $\gr_{\nu}(I) = I_{\nu}/I_{\nu^+}$ for $\nu \in \frac{1}{h}\bbZ_{\geq 1}$.  To this end, we fix an element $\xi \in k_F$ for which $\bbF_p(\xi) = k_F$, as well as a multiplicative lift $[\xi] \in \mathcal{O}_F$.  

\begin{itemize}
\item Suppose first that $n \in \bbZ_{\geq 1}$, and consider the graded piece $\gr_n(I)$.  Using equations \eqref{eqn:gr-integral1} and \eqref{eqn:gr-integral2}, and the fact that $T_n$ normalizes the subgroups $u_\gamma(p^m\cO_F)$, we see that
\begin{equation}\label{eq: basis-integral}
    \gr_n(I) = I_n/I_{n^+} \cong T_{n}/T_{n + 1}.
\end{equation}
(In effect, we see that the unipotent elements of $I_{n}$ already lie in $I_{n^+}$.)  A spanning set for this $\bbF_p$-vector space is given by the $\gr(\lambda(1 + p^n[\xi]^r))$, as $\lambda$ ranges over a basis for $X_*(\mathbf{T})$ and $0 \leq r \leq f - 1$.  Moreover, the image of $\lambda(1 + p^n\cO_F)$ in $\gr_n(I)$ has the structure of a one-dimensional $k_F$-vector space in the natural way.

\item Suppose now that $n \in \bbZ_{\geq 0}$ and $k\in \{1,2,\ldots,h-1\}$. We claim that a basis for the $\F_p$-vector space $\gr_{n + \frac{k}{h}}(I)$ is given by 
\begin{equation}
    \label{eqn:basis-nonintegral}
    \left\{\gr(u_{\beta}(p^{n+1}[\xi]^r)),~\gr(u_{\alpha}(p^n[\xi]^r)) ~:~ \beta\in\Phi_k^-, \alpha\in\Phi_k^+, 0 \leq r \leq f - 1\right\}.
\end{equation}
Indeed, let $u_1, \dots, u_s$ denote the elements $u_{\beta}(p^{n+1}[\xi]^r),~u_{\alpha}(p^n[\xi]^r)$ for $\beta\in\Phi_k^-, \alpha\in\Phi_k^+,$ and $0 \leq r \leq f - 1$, ordered according to the total order on $\Phi$.  The uniqueness of the decomposition in \eqref{eqn:gr-nonintegral1} implies that for any sequence $a_1, a_2, \ldots, a_s \in \{0,1,\ldots, p - 1\}$, not all zero, the element $u_1^{a_1} \cdots u_s^{a_s}$ does not lie in  $I_{(n+\frac{k}{h})^+}$.  Thus, the set \eqref{eqn:basis-nonintegral} is linearly independent.  On the other hand, using normality of $I_{\nu^+}$ in $I_\nu$ shows that any element of $\gr_{n + \frac{k}{h}}(I) = I_{n + \frac{k}{h}}/I_{(n + \frac{k}{h})^+}$ can be written as $gI_{(n + \frac{k}{h})^+}$ with $g \in \prod_{\beta\in\Phi_k^-}u_{\beta}(p^{n+1}\cO_F) \cdot \prod_{\alpha\in\Phi_k^+}u_{\alpha}(p^n\cO_F)$.  This means \eqref{eqn:basis-nonintegral} spans $\gr_{n + \frac{k}{h}}(I)$, and implies that as $\bbF_p$-vector spaces we have
\[
\gr_{n+\frac{k}{h}}(I) = \bigoplus_{\beta\in\Phi_k^-}u_{\beta}(p^{n+1}\cO_F/p^{n+2}\cO_F)\oplus \bigoplus_{\alpha\in\Phi_k^+}u_{\alpha}(p^n\cO_F/p^{n+1}\cO_F).
\]
As above, the image of $u_\beta(p^{n + 1}\cO_F)$ for $\beta \in \Phi^-_k$ (resp., $u_\alpha(p^n\cO_F)$ for $\alpha \in \Phi^+_k$) in $\gr_{n + \frac{k}{h}}(I)$ is naturally endowed with the structure of a one-dimensional $k_F$-vector space.  
\end{itemize}

\subsection{\texorpdfstring{The operator $P$}{The operator P}}
\label{subsec:P-operator}

Next, we calculate the operator $P:\gr_\nu(I)\longrightarrow\gr_{\nu+1}(I)$, defined by $iI_{\nu^+} \longmapsto i^pI_{(\nu + 1)^+}$.  

\begin{itemize}
\item  Suppose first that $n \in \bbZ_{\geq 1}$, so that $\gr_n(I) \cong T_n/T_{n + 1}$, and let $\lambda \in X_*(\mathbf{T})$.  Applying the operator $P$ to $\gr(\lambda(1 + p^n[\xi]^r))$ in $\gr_n(I)$ gives
\[
P\left(\lambda(1+p^n[\xi]^r)I_{n^+}\right) = \lambda(1+p^{n}[\xi]^r)^pI_{(n + 1)^+} = \lambda(1+p^{n + 1}[\xi]^r)I_{(n + 1)^+}.
\]
Thus, we see that $P$ induces an $\bbF_p$-linear isomorphism $P: \gr_n(I) \stackrel{\sim}{\longrightarrow} \gr_{n + 1}(I)$.

\item Suppose now that $n \in \bbZ_{\geq 0}$ and $k\in \{1,2,\ldots,h-1\}$. Applying the operator $P$ to the basis elements \eqref{eqn:basis-nonintegral} gives
\[
P\left(u_{\alpha}(p^n[\xi]^r)I_{(n + \frac{k}{h})^+}\right) = u_{\alpha}(p^{n}[\xi]^r)^pI_{(n + \frac{k}{h}+1)^+} = u_{\alpha}(p^{n+1}[\xi]^r)I_{(n + \frac{k}{h}+1)^+}
\]
for $\alpha\in\Phi_k^+$, and analogously for $\beta \in \Phi^-_k$. In particular, we see that $P$ maps basis elements of $\gr_{n + \frac{k}{h}}(I)$ to basis elements of $\gr_{n + \frac{k}{h} + 1}(I)$, and induces an $\bbF_p$-linear isomorphism $P:\gr_{n + \frac{k}{h}}(I) \stackrel{\sim}{\longrightarrow}\gr_{n + \frac{k}{h} + 1}(I)$.  
\end{itemize}
Furthermore, by the explicit descriptions above, we see that the operator $P$ is actually $k_F$-linear.

\subsection{Lie bracket}

Our final task will be to compute Lie brackets of basis elements of $\gr(I)$. We recall that these Lie brackets are induced by the commutator of elements in $I$. We also recall the structure constants $c_{\alpha, \beta;i,j}$ defined by the commutator formula \eqref{commutator}.

In what follows, we use the following notation: for $\gamma \in \Phi$, we set
$$\delta_\gamma := \begin{cases} 0 & \textnormal{if $\gamma \in \Phi^+$,} \\ 1 & \textnormal{if $\gamma \in \Phi^-$}.\end{cases}$$

\begin{prop}
\label{prop:comm1}
Fix $\alpha, \beta \in \Phi$ with $\beta \neq \pm \alpha$.  Let $n,m \geq 0$, and choose $x_\alpha, x_\beta \in \cO_F^\times$, so that $\gr(u_\alpha(p^{n + \delta_\alpha}x_\alpha))$ is an element of $\gr_{n + \frac{k}{h}}(I)$ for some $0 < k < h$ (and likewise for $\beta$ we get an element in $\gr_{m+\frac{\ell}{h}}(I)$ with some $0 < \ell < h$).  Then, in $\gr_{n + m + \frac{k+\ell}{h}}(I)$, we have
\[
\left[\gr(u_{\alpha}(p^{n + \delta_\alpha}x_\alpha)),\gr(u_{\beta}(p^{m + \delta_\beta}x_\beta))\right] = \begin{cases}\gr(u_{\alpha + \beta}(c_{\alpha,\beta;1,1}p^{n + m + \delta_\alpha + \delta_\beta}x_\alpha x_\beta)), & \textnormal{if $\alpha+\beta \in \Phi$},\\ 0 & \textnormal{otherwise}.\end{cases}
\]
\end{prop}

\begin{proof}
    We may assume $\alpha + \beta \in \Phi$, otherwise there is nothing to prove.  By the commutator formula \eqref{commutator}, we have
    \[
    \left[u_{\alpha}(p^{n + \delta_\alpha}x_\alpha),u_{\beta}(p^{m + \delta_\beta}x_\beta)\right] = \prod_{\substack{i,j > 0 \\ i\alpha + j\beta \in \Phi}}u_{i\alpha + j\beta}\left(c_{\alpha,\beta;i,j}p^{i(n + \delta_\alpha) + j(m + \delta_\beta)}x_\alpha^i x_\beta^j\right).
    \]
  We compute the $p$-valuation of a term $u_{i\alpha + j\beta}(c_{\alpha,\beta;i,j}p^{i(n + \delta_\alpha) + j(m + \delta_\beta)}x_\alpha^i x_\beta^j)$ appearing above. Note that the height of $\alpha$ is equal to $k - h\delta_\alpha$ (and similarly for $\beta$).  Then the height of $i\alpha + j\beta$ is $i(k - h\delta_\alpha) +j(\ell - h\delta_\beta)$, and thus
    \begin{eqnarray*}
    \omega\left(u_{i\alpha + j\beta}(c_{\alpha,\beta;i,j}p^{i(n + \delta_\alpha) + j(m + \delta_\beta)}x_\alpha^i x_\beta^j)\right) & = & i(n + \delta_\alpha) + j(m + \delta_\beta) + \frac{i(k - h\delta_\alpha) + j(\ell - h\delta_\beta)}{h} \\
    & = & in + jm + \frac{ik + j\ell}{h}
    \end{eqnarray*}
    (we are using that $\textnormal{val}_p(c_{\alpha,\beta;i,j}) = 0$ by Lemma \ref{lem:boundconsq's}).  However,  
    \[
    in + jm + \frac{ik + j\ell}{h} \geq n + m +\frac{k + \ell}{h},
    \]
    with equality if and only if $i = j =1$. Therefore the image of $[u_{\alpha}(p^{n + \delta_\alpha}x_\alpha),u_{\beta}(p^{m + \delta_\beta}x_\beta)]$ in $\gr_{n + m+\frac{k + \ell}{h}}(I)$ is as claimed.
\end{proof}

\begin{prop}
\label{prop:comm2}
    Fix $\alpha \in \Phi_k^+$.  Let $n,m \geq 0$, and choose $x_\alpha, x_{-\alpha} \in \cO_F^\times$, so that $\gr(u_\alpha(p^{n}x_\alpha))$ is an element of $\gr_{n + \frac{k}{h}}(I)$ and $u_{-\alpha}(p^{m + 1}x_{-\alpha})$ is an element of $\gr_{m + \frac{h-k}{h}}(I)$. Then, in $\gr_{n + m + 1}(I)$, we have
    \[
    \left[\gr(u_\alpha(p^{n}x_\alpha)), \gr(u_{-\alpha}(p^{m+1}x_{-\alpha}))\right] = \gr(\alpha^\vee(1 + p^{n + m +1}x_\alpha x_{-\alpha})).
    \]
\end{prop}

\begin{proof}
    Pulling back via the morphism $\varphi_{\alpha}$, we are reduced to a calculation in $\mathrm{SL}_2$ (or, equivalently, we may invoke \cite[Thm. 4.2.6]{conrad:reductive}).  Thus, we compute:
    \[
    \left[\begin{pmatrix}
        1 & p^{n}x_\alpha \\ 0 & 1
    \end{pmatrix}, \begin{pmatrix}
        1 & 0 \\ p^{m+1}x_{-\alpha} & 1
    \end{pmatrix}\right] = \begin{pmatrix}
        1+p^{n + m + 1}x_\alpha x_{-\alpha} + p^{2(n + m + 1)}x_\alpha^2 x_{-\alpha}^2 & -p^{2n + m + 1}x_\alpha^2 x_{-\alpha} \\ p^{n+ 2(m+1)}x_\alpha x_{-\alpha}^2 & 1 - p^{n + m + 1}x_\alpha x_{-\alpha}
    \end{pmatrix}.
    \]
    In order to calculate the $p$-valuation of the right-hand side, we factor it as follows:
   \begin{multline}
   \label{eqn:alpha-alphacomm}
    \begin{pmatrix} 1 & p^n x_\alpha\left(1 - (1 - p^{n + m + 1}x_\alpha x_{-\alpha})^{-1}\right) \\ 0 & 1 \end{pmatrix} \begin{pmatrix} (1 - p^{n + m + 1}x_\alpha x_{-\alpha})^{-1} & 0 \\ 0 & 1 - p^{n + m + 1}x_\alpha x_{-\alpha}\end{pmatrix} \\
    \cdot \begin{pmatrix} 1 & 0 \\p^{m + 1} x_{-\alpha}\left((1 - p^{n + m + 1}x_\alpha x_{-\alpha})^{-1} - 1\right) & 1\end{pmatrix}.
    \end{multline}
    Since 
\[
(1 - p^{n + m + 1}x_\alpha x_{-\alpha})^{-1} - 1 = \sum_{k \geq 1} p^{(n + m + 1)k}x_\alpha^k x_{-\alpha}^k,
\]
we see that after applying $\varphi_\alpha$, the left and right factors of $\varphi_\alpha(\eqref{eqn:alpha-alphacomm})$ will have high $p$-valuations, namely
\begin{eqnarray*}
\omega\Big(u_\alpha\left(p^n x_\alpha\left(1 - (1 - p^{n + m + 1}x_\alpha x_{-\alpha})^{-1}\right)\right)\Big) & = & 2n + m + 1 + \frac{k}{h}, \\
\omega\Big( u_{-\alpha}\left(p^{m + 1} x_{-\alpha}\left((1 - p^{n + m + 1}x_\alpha x_{-\alpha})^{-1} - 1\right)\right)\Big) & = & n + 2(m + 1) - \frac{k}{h}.
\end{eqnarray*}
Both of these quantities are strictly larger than $n + m + 1$, which shows that both the left and right factors of $\varphi_\alpha(\eqref{eqn:alpha-alphacomm})$ lie in $I_{(n + m + 1)^+}$.  Further, after applying $\varphi_\alpha$, the middle term of $\varphi_\alpha(\eqref{eqn:alpha-alphacomm})$ has $p$-valuation
\[\omega\left(\alpha^\vee(1 - p^{n + m + 1}x_\alpha x_{-\alpha})^{-1}\right) = n + m + 1.\]
Hence, in $\gr_{n + m + 1}(I)$, we have
\[[\gr(u_{\alpha}(p^n x_\alpha)),~ \gr(u_{-\alpha}(p^{m + 1} x_{-\alpha}))] = \gr\left(\alpha^\vee(1 - p^{n + m + 1}x_\alpha x_{-\alpha})^{-1}\right) = \gr\left(\alpha^\vee(1 + p^{n + m + 1}x_\alpha x_{-\alpha})\right),\]
and we are done.    
\end{proof}

The following result is not strictly necessary to deduce the structure $\gr(I)\otimes_{\bbF_p[P]}\bbF_p$ below, but we include it here for the sake of completeness.

\begin{prop}
\label{prop:comm3}
Fix $\alpha\in \Phi$ and $\lambda \in X_*(\mathbf{T})$ such that $p\nmid \langle\alpha,\lambda \rangle$.  Let $n\geq 0$, $m \geq 1$, and choose $x_\alpha, y \in \cO_F^\times$, so that $\gr(u_\alpha(p^{n + \delta_\alpha} x_\alpha))$ is an element of $\gr_{n + \frac{k}{h}}(I)$ for some $0 < k < h$ and $\gr(\lambda(1 + p^m y))$ is an element of $\gr_{m}(I)$.  Then, in $\gr_{n + m + \frac{k}{h}}(I)$, we have
\[
\left[\gr(\lambda(1 + p^m y)),~ \gr(u_\alpha(p^{n + \delta_\alpha} x_\alpha))\right] = \gr\left(u_{\alpha}(\langle \alpha,\lambda\rangle p^{n + \delta_\alpha + m} x_\alpha y)\right).
\]
\end{prop}

\begin{proof}
The conjugation action of $\mathbf{T}$ on the root space $\mathbf{U}_\alpha$ implies that 
\begin{eqnarray*}
[\lambda(1 + p^m y), u_\alpha(p^{n + \delta_\alpha} x_\alpha)] & = & \lambda(1 + p^m y)u_\alpha(p^{n + \delta_\alpha} x_\alpha)\lambda(1 + p^m y)^{-1}u_\alpha(p^{n + \delta_\alpha} x_\alpha)^{-1}\\
&  = & u_{\alpha}\left(p^{n + \delta_\alpha} x_\alpha(\alpha(\lambda(1 + p^m y)) - 1)\right)\\
& = & u_{\alpha}\left(p^{n + \delta_\alpha} x_\alpha((1 + p^m y)^{\langle \alpha, \lambda\rangle} - 1)\right),
\end{eqnarray*}
as elements of $I$.  The identity
\[
(1 + p^m y)^{\langle \alpha, \lambda\rangle} - 1 = \sum_{k \geq 1} \binom{\langle \alpha,\lambda\rangle}{k}p^{mk}y^k,
\]
implies that we have 
$$\omega \left(u_{\alpha}(p^{n + \delta_\alpha} x_\alpha((1 + p^m y)^{\langle \alpha, \lambda\rangle} - 1))\right) = n + \delta_\alpha + \frac{\hgt(\alpha)}{h} + m = n + \delta_\alpha + \frac{k - h\delta_\alpha}{h} + m = n + m + \frac{k}{h},$$
where we have used that $p$ does not divide $\langle \alpha,\lambda\rangle$.  This gives the claim.  
\end{proof}

Combining the above results, we obtain a presentation for the graded Lie algebra $\gr(I)$.  For the discussion that follows, we fix a basis $\cB$ of $X_*(\mathbf{Z})$.  

\begin{cor}
\label{cor:gr(I)pres}
An $\F_p$-basis for the graded Lie algebra $\gr(I)$ is given by the following elements:
\begin{itemize}
\item $\gr(u_\alpha(p^{n + \delta_\alpha}[\xi]^r))$, where $\alpha \in \Phi$, $n \geq 0$ and $0 \leq r \leq f - 1$, and
\item $\gr(\lambda(1 + p^m[\xi]^r))$, where $\lambda \in \cB \cup \Delta^\vee$, $m \geq 1$ and $0 \leq r \leq f - 1$.
\end{itemize}
The Lie brackets between these elements are determined by Propositions \ref{prop:comm1}, \ref{prop:comm2}, and \ref{prop:comm3}.
\end{cor}

\begin{proof}
It only remains to calculate the bracket $[\gr(\lambda(1 + p^m [\xi]^{r})),~ \gr(u_\alpha(p^{n + \delta_\alpha} [\xi]^s))]$ when $p$ divides $\langle \alpha, \lambda\rangle$.  However, when $\lambda \in \Delta^\vee$, the assumptions $p > h + 1$ and $p | \langle \alpha, \lambda\rangle$ imply $\langle \alpha, \lambda\rangle = 0$, while if $\lambda \in \cB$, then $\langle \alpha, \lambda\rangle = 0$ follows from centrality of $\mathbf{Z}$.  Therefore, $\lambda(1 + p^m [\xi]^{r})$ commutes with $u_\alpha(p^{n + \delta_\alpha} [\xi]^s)$, and the desired bracket is 0.
\end{proof}

In fact, Propositions \ref{prop:comm1}, \ref{prop:comm2}, and \ref{prop:comm3} give us a bit more: namely, that the Lie bracket on $\gr(I)$ is actually $k_F$-bilinear.  Hence, $\gr(I)$ becomes a graded Lie algebra over $k_F$.  More precisely, if we let $\fI$ denote the $\bbF_p$-span in $\gr(I)$ of $\gr(\lambda(1 + p^m))$ for $\lambda \in \cB \cup \Delta^\vee$ and $m \geq 1$, and the elements $\gr(u_\alpha(p^{n + \delta_\alpha}))$ for $\alpha \in \Phi$ and $n \geq 0$, then we we get
\begin{equation} \label{eq:kF-lie-algebras}
\gr(I) \cong k_F \otimes_{\bbF_p}\fI
\end{equation}
as graded $k_F$-Lie algebras. Moreover, recall from Subsection \ref{subsec:P-operator} that the operator $P$ is $k_F$-linear, and from Subsection \ref{subsec:p-valuations} that $\gr(I)$ is an $\F_p[P]$-Lie algebra (free as an $\F_p[P]$-module, see \cite[Rmk. 25.2]{schneider:padicliegroups}). It follows that (\ref{eq:kF-lie-algebras}) is in fact an isomorphism of $k_F[P]$-Lie algebras (free as $k_F[P]$-modules).

\section{\texorpdfstring{The universal enveloping algebra of $\overline{\gr(I)}$}{The universal enveloping algebra of \overline{gr(I)}}}

In this section, we write down a presentation of the graded Lie algebra $\overline{\gr(I)} := \gr(I)\otimes_{\bbF_p[P]}\bbF_p$ in terms of a basis and Lie bracket relations.  We use this presentation to describe the universal enveloping algebra $U_{\bbF_p}(\overline{\gr(I)})$ and calculate its maximal commutative quotient.  In what follows, given an element $i \in I \smallsetminus\{1\}$, we use the notation $\overline{i}$ to denote $\gr(i) \otimes 1 \in \gr(I) \otimes_{\bbF_p[P]} \bbF_p = \overline{\gr(I)}$.

\subsection{\texorpdfstring{Presentation of $\overline{\gr(I)}$}{Presentation of \overline{gr(I)}}}

Note first that, by the calculation of the $P$ operator in Subsection \ref{subsec:P-operator}, we have an isomorphism of $\bbF_p$- (or $k_F$-)vector spaces:
\begin{equation}
\label{eqn:gr(I)-bar-grading}
\overline{\gr(I)} = \gr(I)\otimes_{\bbF_p[P]}\bbF_p\cong \gr_{\frac{1}{h}}(I)\oplus\gr_{\frac{2}{h}}(I)\oplus...\oplus\gr_{\frac{h-1}{h}}(I)\oplus\gr_1(I).
\end{equation}
The above decomposition gives the grading on the Lie algebra $\overline{\gr(I)}$. As we have done at the end of the previous section, $\mathcal{B}$ denotes a fixed basis of $X_*(\mathbf{Z})$. Using Corollary \ref{cor:gr(I)pres}, we obtain the following presentation:

\begin{prop}
\label{prop:overline-gr(I)-pres}
A basis for the graded Lie algebra $\overline{\gr(I)} = \gr(I)\otimes_{\bbF_p[P]}\bbF_p$ is given by the following elements:
\begin{itemize}
\item $\overline{u_\alpha(p^{\delta_\alpha}[\xi]^r)}$, where $\alpha \in \Phi$ and $0 \leq r \leq f - 1$, and
\item $\overline{\lambda(1 + p[\xi]^r)}$, where $\lambda \in \cB \cup \Delta^\vee$ and $0 \leq r \leq f - 1$.
\end{itemize}
The elements $\overline{\lambda(1 + p[\xi]^r)}$ are central, and the Lie brackets between two elements of the form $\overline{u_{\alpha}(p^{\delta_\alpha}[\xi]^r)}$ are determined by Propositions \ref{prop:comm1} and \ref{prop:comm2}.  In particular, if $\alpha,\beta \in \Phi^+$, then we have
$$\left[\overline{u_{\alpha}([\xi]^r)},~\overline{u_{\beta}([\xi]^s)}\right] = \begin{cases}\overline{u_{\alpha + \beta}(c_{\alpha,\beta;1,1}[\xi]^{r + s})}, & \textnormal{if $\alpha+\beta \in \Phi$},\\ 0 & \textnormal{otherwise},\end{cases}$$
while if $\alpha,\beta \in \Phi^-$, then we have
$$\left[\overline{u_{\alpha}(p[\xi]^r)},~\overline{u_{\beta}(p[\xi]^s)}\right] = 0.$$
\end{prop}

\subsection{The universal enveloping algebra}

We begin with some generalities.  Suppose $\mathfrak{g}$ is a finite-dimensional Lie algebra over $\mathbb{F}_p$ (though the arguments below work for any coefficient field), and let $[\mathfrak{g},\mathfrak{g}]$ denote its derived Lie subalgebra.  The surjection $\mathfrak{g} \longtwoheadrightarrow \mathfrak{g}/[\mathfrak{g}, \mathfrak{g}]$ of Lie algebras induces a surjection $U_{\mathbb{F}_p}(\mathfrak{g}) \longtwoheadrightarrow U_{\mathbb{F}_p}(\mathfrak{g}/[\mathfrak{g}, \mathfrak{g}])$ of associative $\mathbb{F}_p$-algebras.  Since $\mathfrak{g}/[\mathfrak{g}, \mathfrak{g}]$ is an abelian Lie algebra, $U_{\mathbb{F}_p}(\mathfrak{g}/[\mathfrak{g}, \mathfrak{g}])$ is a commutative $\bbF_p$-algebra (precisely, it is a polynomial algebra with generators given by an $\mathbb{F}_p$-basis for $\mathfrak{g}/[\mathfrak{g}, \mathfrak{g}]$).  In particular, the surjection $U_{\mathbb{F}_p}(\mathfrak{g}) \longtwoheadrightarrow U_{\mathbb{F}_p}(\mathfrak{g}/[\mathfrak{g}, \mathfrak{g}])$ factors through the largest commutative quotient $U_{\mathbb{F}_p}(\mathfrak{g})^{\textnormal{ab}}$ of $U_{\mathbb{F}_p}(\mathfrak{g})$:
$$U_{\mathbb{F}_p}(\mathfrak{g}) \longtwoheadrightarrow U_{\mathbb{F}_p}(\mathfrak{g})^{\textnormal{ab}} \longtwoheadrightarrow U_{\mathbb{F}_p}(\mathfrak{g}/[\mathfrak{g}, \mathfrak{g}]).$$

We claim that $U_{\mathbb{F}_p}(\mathfrak{g})^{\textnormal{ab}} \longtwoheadrightarrow U_{\mathbb{F}_p}(\mathfrak{g}/[\mathfrak{g}, \mathfrak{g}])$ is an isomorphism.  Indeed, let $J$ denote the commutator ideal of $U_{\mathbb{F}_p}(\mathfrak{g})$ (that is, the two-sided ideal generated by all expressions $ab - ba$, where $a,b \in U_{\mathbb{F}_p}(\mathfrak{g})$), and let $J'$ denote the two-sided ideal of $U_{\mathbb{F}_p}(\mathfrak{g})$ generated by (the image in $U_{\mathbb{F}_p}(\mathfrak{g})$ of) $[\mathfrak{g},\mathfrak{g}]$.  Then $U_{\mathbb{F}_p}(\mathfrak{g})/J \cong U_{\mathbb{F}_p}(\mathfrak{g})^{\textnormal{ab}}$ and $U_{\mathbb{F}_p}(\mathfrak{g})/J' \cong U_{\mathbb{F}_p}(\mathfrak{g}/[\mathfrak{g}, \mathfrak{g}])$ by \cite[Ch. I, \S 2.3, Prop. 3]{bourbaki:lie1}.  The existence of the surjection $U_{\mathbb{F}_p}(\mathfrak{g})^{\textnormal{ab}} \longtwoheadrightarrow U_{\mathbb{F}_p}(\mathfrak{g}/[\mathfrak{g}, \mathfrak{g}])$ implies that $J \subset J'$, so we only need to prove the reverse inclusion.  Moreover, since $J'$ is generated by $[\mathfrak{g}, \mathfrak{g}]$, it suffices to show $[\mathfrak{g}, \mathfrak{g}] \subset J$, which now follows from the definitions.

We therefore conclude:

\begin{lem}
  \label{lemma:liealg1}
  Suppose $\mathfrak{g}$ is a finite-dimensional Lie algebra over $\mathbb{F}_p$.  Then the maximal commutative quotient $U_{\mathbb{F}_p}(\mathfrak{g})^{\textnormal{ab}}$ of $U_{\mathbb{F}_p}(\mathfrak{g})$ is isomorphic to $U_{\mathbb{F}_p}(\mathfrak{g}/[\mathfrak{g}, \mathfrak{g}])$.  In particular, $U_{\mathbb{F}_p}(\mathfrak{g})^{\textnormal{ab}}$ is a polynomial algebra, with generators given by an $\mathbb{F}_p$-basis of $\mathfrak{g}/[\mathfrak{g}, \mathfrak{g}]$.  
\end{lem}

Suppose now that $S \subset \mathfrak{g}$ is a generating set, meaning that the smallest Lie subalgebra of $\mathfrak{g}$ containing $S$ is $\mathfrak{g}$ itself.  We then obtain a surjection of Lie algebras $\mathfrak{L}(S) \longtwoheadrightarrow \mathfrak{g}$, where $\mathfrak{L}(S)$ denotes the free Lie algebra on the set $S$ (see \cite[Ch. II, \S\S 2.1 -- 2.3]{bourbaki:lie2-3}).  This in turn gives a surjection of associative $\mathbb{F}_p$-algebras $U_{\mathbb{F}_p}(\mathfrak{L}(S)) \longtwoheadrightarrow U_{\mathbb{F}_p}(\mathfrak{g})$.  By \cite[Ch. II, \S 3.1, Thm. 1]{bourbaki:lie2-3}, $U_{\mathbb{F}_p}(\mathfrak{L}(S))$ is isomorphic to the free associative algebra on the set $S$.  We deduce:

\begin{lem}
  \label{lemma:liealg2}
  Suppose $\mathfrak{g}$ is a finite-dimensional Lie algebra over $\mathbb{F}_p$, and let $S\subset \mathfrak{g}$ denote a generating set.  Then (the image in $U_{\mathbb{F}_p}(\mathfrak{g})$ of) $S$ generates $U_{\mathbb{F}_p}(\mathfrak{g})$ as an associative $\mathbb{F}_p$-algebra.  Moreover, if $S$ is minimal as a generating set for $\mathfrak{g}$, then $S$ is minimal as a generating set for $U_{\mathbb{F}_p}(\mathfrak{g})$.  
\end{lem}

We now suppose additionally that $\mathfrak{g}$ is a nilpotent Lie algebra, and choose an $\mathbb{F}_p$-basis $\overline{A}$ of the finite-dimensional vector space $\mathfrak{g}/[\mathfrak{g}, \mathfrak{g}]$.  For each $\overline{a} \in \overline{A}$, choose a preimage $a \in \mathfrak{g}$, and denote by $A$ the collection of these preimages.  By \cite[Ch. I, Exers. du \S 4, Exer. 4]{bourbaki:lie1}, $A$ gives a minimal set of generators for $\mathfrak{g}$.  Combining with the previous two lemmas, we obtain the following:

\begin{lem}
\label{lemma:liealg3}
Suppose $\mathfrak{g}$ is a finite-dimensional nilpotent Lie algebra over $\mathbb{F}_p$, let $\overline{A}$ denote an $\mathbb{F}_p$-basis of $\mathfrak{g}/[\mathfrak{g}, \mathfrak{g}]$, and choose a set $A \subset \mathfrak{g}$ of preimages.  Then $A$ is a minimal generating set for $\mathfrak{g}$, and its canonical image in $U_{\mathbb{F}_p}(\mathfrak{g})$ is a minimal generating set for $U_{\mathbb{F}_p}(\mathfrak{g})$.  Moreover, the maximal commutative quotient of $U_{\mathbb{F}_p}(\mathfrak{g})$ is a polynomial algebra, with generators given by $\overline{A}$.
\end{lem}

We now apply the above discussion along with the results of the previous section to the nilpotent Lie algebra $\overline{\gr(I)}$.  A description of $U_{\mathbb{F}_p}(\overline{\gr(I)})$ is given by
\begin{equation}
\label{eqn:univ-env-alg-pres}
    U_{\bbF_p}(\overline{\gr(I)}) \cong \FF_p\left\langle \overline{u_{\gamma}(p^{\delta_\gamma}[\xi]^r)} ,~ \overline{\lambda(1+p[\xi]^r)} ~:~ \gamma\in \Phi, \lambda \in \cB \cup \Delta^\vee,  0\leq r\leq f-1 \right\rangle/\sim,
\end{equation}
where $\sim$ represents the commutation relations determined by the Lie bracket in Proposition \ref{prop:overline-gr(I)-pres}, namely for $i,j \in I$, we have $\overline{i}\cdot \overline{j} - \overline{j}\cdot \overline{i} = [\overline{i}, \overline{j}] = [\gr(i),\gr(j)]\otimes 1$ in $U_{\F_p}(\overline{\gr(I)})$.  The notation above denotes the non-commutative polynomial algebra over $\FF_p$ generated by the specified elements (i.e., the tensor algebra of $\overline{\gr(I)}$), where the relations $\sim$ dictate at what cost we can allow the symbols to commute.  Note also that the grading on $\overline{\gr(I)}$ induces a $\frac{1}{h}\bbZ_{\geq 0}$-grading on $U_{\bbF_p}(\overline{\gr(I)})$.  Combining equation \eqref{eqn:univ-env-alg-pres} with Corollary \ref{lazardcor} yields a presentation for $\gr_\omega(\mathbb{F}_p\llbracket I\rrbracket)$.

In order to apply Lemma \ref{lemma:liealg3}, we compute the abelianization of $\overline{\gr(I)}$.

\begin{lem}
    In terms of the decomposition \eqref{eqn:gr(I)-bar-grading}, we have
    \begin{equation}
    \label{eqn:commutator-ideal}
        \left[\overline{\gr(I)}, \overline{\gr(I)}\right] = \gr_{\frac{2}{h}}(I)\oplus...\oplus\gr_{\frac{h-1}{h}}(I) \oplus \textnormal{span}\left\{\overline{\alpha^\vee(1 + p[\xi]^r)}~:~\alpha \in \Delta, 0 \leq r \leq f - 1\right\}.
    \end{equation}
\end{lem}

\begin{proof}
    Suppose first that $\gamma\in\Phi_k^+$ with $2 \leq k \leq h - 1$, and write $\gamma = \alpha + \beta$ for some $\alpha\in\Phi_1^+$ and $\beta \in \Phi_{k-1}^+$. Then by applying Proposition \ref{prop:comm1}, we get
    \[
    \left[\gr(u_{\alpha}(1)),~\gr(u_{\beta}([\xi]^r))\right] = \gr(u_{\gamma}(c_{\alpha,\beta;1,1}[\xi]^r)) = c_{\alpha,\beta;1,1}\gr(u_{\gamma}([\xi]^r))
    \]
    (for the final equality, recall that $\gr(I)$ is an $\bbF_p$-vector space).  Since $c_{\alpha,\beta;1,1}$ is nonzero in $\mathbb{F}_p$, we see that $\overline{u_\gamma([\xi]^r}) \in [\overline{\gr(I)}, \overline{\gr(I)}]$.  

Suppose next that $\gamma\in \Phi_k^-$ with $2 \leq k \leq h - 1$.  Then $\gamma$ has height $k-h$ and we can write $\gamma = \alpha + \beta$ for some $\alpha\in \Phi_1^+$ and $\beta\in \Phi_{k-1}^-$ (see Lemma \ref{lem: root-addition} below).  Applying Proposition \ref{prop:comm1}, we get 
    \[
    \left[\gr(u_{\alpha}(1)),~ \gr(u_{\beta}(p[\xi]^r))\right] = \gr(u_{\gamma}(c_{\alpha,\beta;1,1}p[\xi]^r)) = c_{\alpha,\beta;1,1}\gr(u_{\gamma}(p[\xi]^r)).
    \]
    As above, we obtain $\overline{u_\gamma(p[\xi]^r}) \in [\overline{\gr(I)}, \overline{\gr(I)}]$.  

  For the elements $\{\overline{\alpha^\vee(1+p[\xi]^r)}:\alpha\in\Delta, 0\leq r\leq f-1\}$, Proposition \ref{prop:comm2} implies
    \[
    \left[\gr(u_\alpha(1)),~\gr(u_{-\alpha}(p[\xi]^r))\right] = \gr(\alpha^\vee(1+p[\xi]^r)),
    \]
from which we conclude $\overline{\alpha^\vee(1 + p[\xi]^r)} \in [\overline{\gr(I)}, \overline{\gr(I)}]$.

The above paragraphs show that the right-hand side of \eqref{eqn:commutator-ideal} is contained in the left-hand side.  To show the reverse inclusion, note that linearity of the Lie bracket implies $[\overline{\gr(I)}, \overline{\gr(I)}]$ is spanned by elements of the form $[\overline{i}, \overline{j}]$, where $\overline{i}$ and $\overline{j}$ are basis elements as in Proposition \ref{prop:overline-gr(I)-pres}.  Moreover, since elements of the form $\overline{\lambda(1 + p[\xi]^r)}$ ($\lambda \in \mathcal{B}\cup\Delta^\vee, 0 \leq r \leq f - 1$) are central, it suffices to consider elements of the form $[\overline{u_\alpha(p^{\delta_\alpha}[\xi]^r)}, \overline{u_\beta(p^{\delta_\beta}[\xi]^s)}]$, where $\alpha,\beta \in \Phi$, $0 \leq r,s \leq f - 1$.  Propositions \ref{prop:comm1} and \ref{prop:comm2} show that this bracket is contained in the right-hand side of \eqref{eqn:commutator-ideal}, which gives the claim.
\end{proof}

\begin{cor}
    The set 
    \begin{equation}
    \label{eqn:min-gen-set}
    \left\{\overline{u_{\gamma}(p^{\delta_\gamma}[\xi]^r)},~\overline{\lambda(1 + p[\xi]^r)}~:~ \gamma\in\Phi_1,\lambda \in \cB, 0 \leq r \leq f - 1\right\}
    \end{equation}
    gives an $\mathbb{F}_p$-basis for $\overline{\gr(I)}/[\overline{\gr(I)}, \overline{\gr(I)}]$.
\end{cor}

Applying Lemma \ref{lemma:liealg3} gives our main structural results.

\begin{cor}
\label{prop: Minimal Generating Set For Universal Enveloping Algebra}
    The set 
    $$\left\{\overline{u_{\gamma}(p^{\delta_\gamma}[\xi]^r)},~\overline{\lambda(1 + p[\xi]^r)}~:~ \gamma\in\Phi_1,\lambda \in \cB, 0 \leq r \leq f - 1\right\}$$
    is a minimal generating set for $U_{\bbF_p}(\overline{\gr(I)})$.  
\end{cor}

\begin{cor}
    \label{prop:largest-comm-quot}
    The largest commutative quotient of $U_{\bbF_p}(\overline{\gr(I)})$ is isomorphic to the polynomial algebra
    $$\FF_p\left[\overline{u_{\gamma}(p^{\delta_{\gamma}}[\xi]^r)},~ \overline{\lambda(1 + p[\xi]^r)}~:~ \gamma\in\Phi_1,\lambda \in \cB, 0\leq r\leq f-1\right].$$
\end{cor}

\section{\texorpdfstring{Comparing filtrations on $\mathbb{F}_p \llbracket I \rrbracket$}{Comparing filtrations on F_p [[I]]}}

We suppose in this section that $\mathbf{G}$ is semisimple, or, equivalently, that $\mathbf{Z} = 1$.  Let $\mfm = \mfm_I$ be the unique maximal ideal of $\F_p \llbracket I \rrbracket$.  Recall that the ring $\F_p \llbracket I \rrbracket$ has two natural filtrations: the $\mfm$-adic filtration, and the filtration induced by $\overline{w}$ (which is associated to the $p$-valuation $\omega$ \cite[I.2.1.7]{lazard}).
We will prove the following generalization of \cite[Prop. 5.3.3]{BHHMS}, comparing the two filtrations.

\begin{prop}
\label{prop: ideal-comparison}
    For $k \in \frac{1}{h} \Z_{\geq 0}$, we have
    \begin{equation*}
        \mfm^{hk} = \{x \in \F_p \llbracket I \rrbracket: \overline{w}(x) \geq k\}. 
    \end{equation*}
    In particular, the graded rings $\gr_\mfm(\F_p \llbracket I \rrbracket)$ and $\gr_\omega(\F_p \llbracket I \rrbracket)$ are isomorphic up to rescaling indices. 
\end{prop}

\subsection{Preparation}

We will need the following general lemma.

\begin{lem}
\label{lem: gp-ring-arithmetic}
    Let $\cG$ be a pro-$p$ group, and let $\mfm$ be the maximal ideal of $\F_p \llbracket \cG \rrbracket$. 
    \begin{enumerate}
        \item For all $g \in \cG$, $(g^p-1) \in \mathfrak{m}^p$.
    \end{enumerate}
    For the rest, let $a, b,c \in \cG$ with $(a-1) \in \mfm^i$, $(b-1) \in \mfm^j$, and $(c-1) \in \mfm^k$ for some $i,j,k \geq 1$. Then:
    \begin{enumerate}
    \setcounter{enumi}{1}
        \item $abc-1 \equiv (ab-1)+(bc-1)+(ac-1)-(a-1)-(b-1)-(c-1) \mod{\mathfrak{m}^{i+j+k}}.$
        \item $ab-1  \equiv (a-1) + (b-1) \equiv ba-1 \mod{\mathfrak{m}^{i+j}}$,
        \item $a^{-1}-1 \equiv -(a-1) \mod{\mathfrak{m}^{2i}}$,
        \item $(bab^{-1}a^{-1}-1) \in \mathfrak{m}^{i+j}$ and $(aba^{-1}b^{-1}-1) \in \mathfrak{m}^{i+j}$.
    \end{enumerate}
\end{lem}

\begin{proof}
    \begin{enumerate}
    \item Since $(g-1) \in \mathfrak{m}$, this follows from the identity $(g-1)^p = g^p-1$ in $\mathbb{F}_p\llbracket \cG\rrbracket$.
        \item This follows from the identity
        \[(a-1)(b-1)(c-1) = (abc-1)-(ab-1)-(bc-1)-(ac-1)+(a-1)+(b-1)+(c-1).\]
        \item For $ab-1$, this follows from the identity
        \[(a-1)(b-1) = (ab-1)-(a-1)-(b-1),\]
        and similarly for $ba-1$.
        \item Setting $b=a^{-1}$ in part 3 shows that $a^{-1} - 1 \in \mfm^i$ and that $a^{-1}-1 \equiv -(a-1) \mod{\mathfrak{m}^{2i}}$.
        \item We may assume that $i \geq j$. We will use parts 2, 3, and 4. First, parts 3 and 4 readily imply that $(bab^{-1} -1) \in \mfm^j$ and $(a^{-1} -1 ) \in \mfm^i$; now, computing modulo $\mathfrak{m}^{i+j}$, we have 
        \begin{align*}
            bab^{-1}a^{-1}-1 &\equiv (bab^{-1}-1)+(a^{-1}-1)\\
            &\equiv [(ba-1)+(ab^{-1}-1)+(bb^{-1}-1)-(b-1)-(a-1)-(b^{-1}-1)] \\
            & \qquad +(a^{-1}-1)\\
            &\equiv [(b-1)+(a-1)]+[(a-1)+(b^{-1}-1)] \\
            & \qquad -(b-1)-(a-1)-(b^{-1}-1)+(a^{-1}-1)\\
            &\equiv (a-1)+ (a^{-1}-1) \\
            &\equiv 0.
        \end{align*}
        Here, the first and third congruences follow from part 3 of the lemma, the second from part 2, and the last from part 4.

        Since $aba^{-1}b^{-1} = (bab^{-1}a^{-1})^{-1}$, it follows from part 4 that $(aba^{-1}b^{-1}-1)\in \mathfrak{m}^{i+j}$ as well.
    \end{enumerate}
    \end{proof}

\begin{lem}
\label{lem: root-addition}
    If $\alpha \in \Phi^+$ is not of maximal height, then there exists a $\delta \in \Delta$ such that $\alpha + \delta \in \Phi^+$.
\end{lem}

\begin{proof}
    We follow the argument in \cite{root-lemma-mo}.   Denote by $\preceq$ the usual partial order on $\Phi^+$. Let $\theta$ be the highest root, so $\theta - \alpha$ is a sum of simple roots with non-negative coefficients (though not necessarily a root itself; see \cite[Ch. VI, \S 1.8, Prop. 25]{bourbaki:lie}). If $\Ht(\theta) = \Ht(\alpha) + 1$, then we are done. Suppose that this is not the case. We claim that if $\alpha \preceq \gamma$ are positive roots with $\Ht(\gamma) - \Ht(\alpha) \geq 2$, then there is a root $\beta \in \Phi^+$ with $\alpha \prec \beta \prec \gamma$. This is sufficient, since $\Ht(\theta) - \Ht(\alpha) \geq 2$, so it follows by induction that there is a root $\beta \succeq \alpha$ with $\Ht(\beta) = \Ht(\alpha) + 1$, and we can take $\delta = \beta - \alpha$.

    Suppose then that $\alpha \preceq \gamma$ with $\Ht(\gamma) - \Ht(\alpha) \geq 2$, and set $v = \gamma - \alpha$, a sum of simple roots with non-negative coefficients. Let $\delta \in \Delta$ be a simple root appearing in the sum, and suppose for the sake of contradiction that neither $\alpha + \delta$ nor $\gamma - \delta$ is a root. Let $(\cdot, \cdot)$ be a positive definite inner product on $\R \Phi$ invariant under the Weyl group. Then $(\alpha, \delta) \geq 0$ and $(\gamma, \delta) \leq 0$ \cite[Ch. VI, \S 1.3, Cor. to Thm. 1]{bourbaki:lie}, so $(v, \delta) \leq 0$. Summing this inequality over all $\delta$ occurring in $v$ gives $(v,v) \leq 0$; however $(\cdot, \cdot)$ is positive definite, so we get a contradiction. 
\end{proof}

Recall the notation 
\begin{equation*}
    \Phi_k = \{\gamma \in \Phi: \textnormal{ht}(\gamma) \equiv k~(\textnormal{mod}~h)\}
\end{equation*}
and 
$$\delta_\alpha = \begin{cases}0 & \textnormal{if $\alpha \in \Phi^+$,} \\ 1 & \textnormal{if $\alpha \in \Phi^-$.} \end{cases}$$

\begin{lem}
\label{lem: roots-power-of-m}
    Let $k \in \{1, \dots, h-1\}$, $z \in \cO_F$, and $\alpha \in \Phi_k$.  Then $u_\alpha(p^{\delta_\alpha}z) - 1 \in \mfm^k$.  
\end{lem}

\begin{proof}
    We first consider the positive roots $\alpha$ and prove that $\Ht(\alpha) \geq k$ implies $u_\alpha(z) - 1 \in \mfm^k$ using a nested induction, first on $k$, then in descending order on $\Ht(\alpha)$. If $k = 1$ then the statement is clear. Now suppose $k > 1$ and assume the statement holds for $k-1$. Suppose $\Ht(\alpha) \geq k$ and write $\alpha = \alpha' + \beta$ where $\beta \in \Delta$ and $\alpha' \in \Phi^+$. The commutator formula \eqref{commutator} for $u_{\alpha'}(x)$ ($x \in \cO_F$) and $u_\beta(1)$ reads
    \begin{equation}\label{eq: commutator-filtration}
        [u_{\alpha'}(x), u_\beta(1)] = u_{\alpha}(c_{\alpha', \beta;1,1} x) \cdot \prod_{\gamma \in S} u_\gamma(z_\gamma) , 
    \end{equation}
     where $\gamma$ ranges over some (possibly empty) subset $S \subset \Phi^+$ with $\Ht(\gamma) > \Ht(\alpha)$ and where $z_\gamma \in \cO_F$. (We note that $\alpha' \neq \beta$ since $\Phi$ is a reduced root system, so the formula applies.) For any $z \in \cO_F$, we can choose $x \in \cO_F$ such that $c_{\alpha', \beta;1,1} x = z$ since $c_{\alpha', \beta;1,1} \in \cO_F^\times$ by Lemma \ref{lem:boundconsq's}. Observe that $u_{\alpha'}(x) -1 \in \mfm^{k-1}$ by the induction hypothesis. If $\alpha$ is of maximal height ($\Ht(\alpha) = h-1$), then $S = \emptyset$, so (\ref{eq: commutator-filtration}) immediately implies that $u_\alpha(z) - 1 \in \mfm^k$ using Lemma \ref{lem: gp-ring-arithmetic}, part 5. Now assume $k \leq \Ht(\alpha) < h-1$. For all $\gamma \in S$, since $\Ht(\gamma) > \Ht(\alpha)$, we have $u_\gamma(z_\gamma) - 1 \in \mfm^k$ by induction hypothesis. Reducing (\ref{eq: commutator-filtration}) modulo $\mfm^k$ and using Lemma \ref{lem: gp-ring-arithmetic} then shows that $u_\alpha(z) - 1 \in \mfm^k$ as well. 

    For negative roots $\alpha$, we will prove that $\Ht(\alpha) \geq k-h$ implies $u_\alpha(pz) - 1 \in \mfm^{k}$ using a similar induction. If $k =1$ then this is clear. For $k > 1$, assume $\Ht(\alpha) \geq k-h$ (so $\alpha$ is not of minimal height) and write $\alpha = \alpha' + \beta$ for some $\beta \in \Delta$ and $\alpha' \in \Phi^-$ using Lemma \ref{lem: root-addition}. Observe that in the commutator formula (with $x \in \cO_F$)
    \begin{equation}\label{eq: negative-commutator-formula}
        [u_{\alpha'}(px), u_{\beta}(1)] = \prod_{\substack{i,j > 0 \\ i\alpha' + j\beta \in \Phi}}u_{i\alpha' + j\beta}(c_{\alpha',\beta;i,j}p^i x^i),
    \end{equation}
    all factors where $\gamma \coloneqq i\alpha' + j\beta$ is a positive root are congruent to $1$ modulo $\mfm^p$: indeed, they are of the form $u_{\gamma}(py) = u_\gamma(y)^p$ for some $y \in \cO_F$, and $u_\gamma(y) \in I$ for $\gamma \in \Phi^+$, so $u_\gamma(y)^p - 1 \in \mfm^p$ by Lemma \ref{lem: gp-ring-arithmetic}. The same is true for any factors with $i > 1$, including ones corresponding to negative roots, but this time with $y \in p\cO_F$. Since $p > h+1 \geq k$ by assumption, we see that reducing (\ref{eq: negative-commutator-formula}) modulo $\mfm^k$ gives 
    \begin{equation}
        [u_{\alpha'}(px), u_{\beta}(1)] \equiv u_{\alpha}(c_{\alpha',\beta;1,1}px) \cdot \prod_{\gamma \in S} u_\gamma(pz_\gamma) \mod \mfm^k
    \end{equation}
    where $\gamma$ ranges over some subset $S \subset \Phi^-$ with $\Ht(\gamma) > \Ht(\alpha)$ and where $z_\gamma \in \cO_F$. Now the rest of the argument proceeds as in the first paragraph by descending induction on $\Ht(\alpha)$, noting that $S = \emptyset$ if $\Ht(\alpha) = -1$. 
\end{proof}

\begin{lem}\label{lem: coroots-power-of-m}
    If $\alpha \in \Phi^+$ and $0 \leq r \leq f-1$, then $\alpha^\vee(1+p[\xi]^r) -1\in \mfm^h$.
\end{lem}
\begin{proof}
    For $0 \leq r \leq f - 1$, define
$$E_r := \begin{pmatrix}
            1 & [\xi]^r \\
            0 & 1 
        \end{pmatrix}, \quad F_r := \begin{pmatrix}
            1 & 0 \\
            -p[\xi]^r & 1 
        \end{pmatrix}, \quad H_r := \begin{pmatrix}
            1+p[\xi]^r  & 0 \\
            0 & (1+p[\xi]^r)^{-1}
        \end{pmatrix}$$
    in $\text{SL}_2(\cO_F)$.  We then have 
    \begin{equation*}
        F_0E_rF_0^{-1} E_r^{-1} = H_r \begin{pmatrix}
            1 & 0\\ 
            -p[\xi]^r(1+p[\xi]^r) & 1 
        \end{pmatrix}^p \begin{pmatrix}
            1 & -[\xi]^{2r}(1+p[\xi]^r)^{-1} \\
            0 & 1 
        \end{pmatrix}^p.
    \end{equation*}
    Applying $\varphi_{\alpha}$ to both sides then gives 
    \begin{equation*}
        [u_{-\alpha}(-p), u_{\alpha}([\xi]^r)] = \alpha^\vee\left(1+p[\xi]^r\right) \cdot u_{-\alpha}\left(-p[\xi]^r(1+p[\xi]^r)\right)^p\cdot u_\alpha\left(-[\xi]^{2r}(1+p[\xi]^r)^{-1}\right)^p
    \end{equation*}
    in $I$. Note that if $k = \Ht(\alpha)$, then $u_{\alpha}([\xi]^r) - 1 \in \mfm^k$ and $u_{-\alpha}(-p) - 1 \in \mfm^{h-k}$ by Lemma \ref{lem: roots-power-of-m}. Using parts 1, 3, and 5 of Lemma \ref{lem: gp-ring-arithmetic}, along with the assumption $p > h+1$, the equation above implies that $\alpha^\vee(1+p[\xi]^r) - 1 \in \mfm^h$. 
\end{proof}

\subsection{\texorpdfstring{Proof of Proposition \ref{prop: ideal-comparison}}{Proof of Proposition 5.1}}

We are now in a position to prove Proposition \ref{prop: ideal-comparison}.  Recall that we have assumed $\mathbf{Z} = 1$, so $\cB = \emptyset$.

\begin{proof}[Proof of Proposition \ref{prop: ideal-comparison}]
    In Subsection \ref{sec: graded-pieces-structure} we gave a basis for the $\F_p$-vector space $\bigoplus_{k=1}^h \gr_{\frac{k}{h}}(I)$, namely the set 
    \begin{equation*}
        \left\{\gr(\alpha^\vee(1+p[\xi]^r)):\alpha \in \Delta, 0 \leq r \leq f - 1 \right\} \cup  \bigcup_{k=1}^{h-1}  \left\{\gr(u_{\gamma}(p^{\delta_\gamma}[\xi]^r)): \gamma \in \Phi_k, 0 \leq r \leq f - 1\right\}.
    \end{equation*}
    Let $x_1, \dots, x_d$ be elements of $I$ such that $\{\gr(x_1), \dots, \gr(x_d)\}$ is equal to the above set. The family $(x_i)_{1\leq i\leq d}$ is an ordered basis of $I$ in the sense of \cite[Def. III.2.2.4, Prop. III.2.2.5]{lazard}, i.e., $(\gr(x_i))_{1 \leq i \leq d}$ is a basis for the $\F_p[P]$-module $\gr(I)$. By the description of $w$ given in \cite[Eqn. III.2.3.8.8]{lazard}, we see that for $k \in \frac{1}{h} \Z_{\geq 0}$, the ideal $\{x \in \F_p \llbracket I \rrbracket: \overline{w}(x) \geq k\}$ is generated by monomials $z^{\alpha} := \prod_{i=1}^d (x_i - 1)^{\alpha_i}$ ($\alpha \in \bbZ_{\geq 0}^d$) such that $\tau(\alpha) \coloneqq \sum_{i=1}^d \omega(x_i) \alpha_i \geq k$. Using Lemmas \ref{lem: roots-power-of-m}  and \ref{lem: coroots-power-of-m}, along with the definition of the $p$-valuation $\omega$, we have that $x_i - 1 \in \mfm^{h\cdot \omega(x_i)}$ for all $i$, and thus $z^{\alpha} \in \mfm^{h \cdot \tau(\alpha)} \subset \mfm^{hk}$. Hence $\{x \in \F_p \llbracket I \rrbracket: \overline{w}(x) \geq k\} \subset \mfm^{hk}$. On the other hand, we see that $\mfm = \{x \in \F_p \llbracket I \rrbracket: \overline{w}(x) \geq 1/h\}$ since $1/h = \min_{g \in I\smallsetminus\{1\}} \omega(g)$, and therefore  
    \begin{equation*}
        \mfm^{hk} = \{x \in \F_p \llbracket I \rrbracket: \overline{w}(x) \geq 1/h\}^{hk}\subset \{x \in \F_p \llbracket I \rrbracket: \overline{w}(x) \geq k\},
    \end{equation*}
    where the last inclusion follows from the properties of the valuation $\overline{w}$.
\end{proof}

\begin{rem}
\label{remark:nonss}
Suppose $\mathbf{G}$ is an arbitrary split reductive group, i.e., that $\mathbf{Z}$ is not necessarily trivial.  Then Proposition \ref{prop: ideal-comparison} no longer holds; indeed, the elements $\lambda(1 + p[\xi]^r) - 1$ for $\lambda \in \cB$ lie in $\mathfrak{m} \smallsetminus \mathfrak{m}^2$.  To see this, note that Lemma \ref{lem:boundconsq's} implies that $I \cong I^{\textnormal{sc}} \times Z_1$, where $I^{\textnormal{sc}}$ denotes the analogously defined Iwahori subgroup of the simply connected cover $\mathbf{G}^{\textnormal{sc}}$ of $\mathbf{G}/\mathbf{Z}$.  We then have 
$$\mathbb{F}_p \llbracket I\rrbracket \cong \mathbb{F}_p \llbracket I^{\textnormal{sc}}\rrbracket \widehat{\otimes}_{\mathbb{F}_p}\mathbb{F}_p \llbracket Z_1\rrbracket \cong \mathbb{F}_p \llbracket I^{\textnormal{sc}}\rrbracket \widehat{\otimes}_{\mathbb{F}_p}\mathbb{F}_p \llbracket X_1, X_2, \ldots, X_{d_Z[F:\bbQ_p]}\rrbracket \cong \mathbb{F}_p \llbracket I^{\textnormal{sc}}\rrbracket  \llbracket X_1, X_2, \ldots, X_{d_Z[F:\bbQ_p]}\rrbracket,$$
where $d_Z := \dim_F(\mathbf{Z})$ and the variables $X_i$ denote the elements $\lambda(1 + p[\xi]^r) - 1$ for $\lambda \in \cB$ and $0 \leq r \leq f - 1$ (see \cite[\S 20]{schneider:padicliegroups}).  

On the other hand, let $J_{Z,\fm}$ denote the two-sided ideal of $\gr_{\fm}(\mathbb{F}_p\llbracket I \rrbracket)$ generated by the images of $\lambda(1 + p[\xi]^r) - 1$ ($\lambda \in \cB, 0 \leq r \leq f - 1$) in $\fm/\fm^2$, and let $J_{Z,\omega}$ denote the two-sided ideal of $\gr_{\omega}(\mathbb{F}_p\llbracket I \rrbracket)$ generated by the images of the $\lambda(1 + p[\xi]^r) - 1$ in $\mathbb{F}_p\llbracket I \rrbracket_1/\mathbb{F}_p\llbracket I \rrbracket_{1^+}$.  Applying Proposition \ref{prop: ideal-comparison} to $I^{\textnormal{sc}}$, we then obtain an isomorphism of graded rings
$$\gr_{\fm}(\mathbb{F}_p\llbracket I \rrbracket)/J_{Z,\fm} \cong \gr_{\omega}(\mathbb{F}_p\llbracket I \rrbracket)/J_{Z,\omega}$$
(after rescaling the grading on the right-hand side by $h$).  
\end{rem}

\section{Relation to Gelfand--Kirillov dimension}

We now consider what the above results imply about Gelfand--Kirillov dimension of smooth representations of $G$.  We refer to \cite[\S 5.1]{BHHMS} for undefined notation and a summary of results.

We first recall the following result from \cite[Lem. 5.1.3]{BHHMS}, whose proof applies verbatim in our more general setting.

\begin{lem}
	\label{lem:GKdim}
	Let $\cG$ be an analytic pro-$p$-group without $p$-torsion, and assume that the graded ring $\gr_{\fm}(\bbF\llbracket \cG\rrbracket)$ is Auslander-regular.  Suppose $x_1, x_2, \ldots, x_m \in \gr_{\fm}(\bbF\llbracket \cG\rrbracket)$ is a sequence of elements satisfying the following conditions:
	\begin{itemize}
		\item $[x_i, \gr_{\fm}(\bbF\llbracket \cG\rrbracket)] \subset \langle x_1, x_2, \ldots, x_{i - 1}\rangle$ for all $1 \leq i \leq m$, the latter denoting the two-sided ideal generated by the given elements,
		\item $x_i$ is neither a left nor right zero-divisor on $\gr_{\fm}(\bbF\llbracket \cG\rrbracket)/\langle x_1, x_2, \ldots, x_{i - 1} \rangle$ for all $1 \leq i \leq m$, and
		\item $\gr_{\fm}(\bbF\llbracket \cG\rrbracket)/\langle x_1, x_2, \ldots, x_{m}\rangle$ is a nonzero commutative ring of finite global dimension, all of whose maximal ideals have the same height.
	\end{itemize}
	Suppose $M$ is a finitely generated $\bbF\llbracket \cG\rrbracket$-module such that $\gr_{\fm}(M)$ is annihilated by $\langle x_1, x_2, \ldots, x_{m}\rangle$. Then $\dim_{\cG}(M)$ is equal to the dimension of the support of $\gr_{\fm}(M)$ in $\textnormal{Spec}(\gr_{\fm}(\bbF\llbracket \cG\rrbracket)/\langle x_1, x_2, \ldots, x_{m}\rangle)$.
\end{lem}

When $\cG = I$ is a pro-$p$ Iwahori subgroup in a $p$-adic reductive group $G$ (with assumptions as in previous sections), the ring $\gr_{\fm}(\bbF\llbracket I\rrbracket)$ is Auslander-regular (see \cite[pf. of Thm. 5.3.4]{BHHMS}).  We obtain the following consequence:

\begin{lem}
	\label{lem:GKbound}
	Suppose $\pi$ is an admissible smooth $G$-representation admitting a central character, and supppose $\gr_{\fm}(\pi^\vee)$ satisfies the conditions of Lemma \ref{lem:GKdim} for some sequence $x_1, x_2, \ldots, x_m \in \gr_{\fm}(\bbF\llbracket I\rrbracket)$.  Then $\dim_G(\pi) \leq [F:\bbQ_p](|\Delta| + 1)$.  
\end{lem}

\begin{proof}
	Suppose $x_1, x_2, \ldots, x_m$ satisfy the conditions of Lemma \ref{lem:GKdim}.  In particular, this implies that $\gr_{\fm}(\bbF\llbracket I\rrbracket)/\langle x_1, x_2, \ldots, x_m\rangle$ is commutative, so the surjection 
	$$\gr_{\fm}(\bbF\llbracket I\rrbracket) \longtwoheadrightarrow \gr_{\fm}(\bbF\llbracket I\rrbracket)/\langle x_1, x_2, \ldots, x_m\rangle$$ 
	then factors through the largest commutative quotient of $\gr_{\fm}(\bbF\llbracket I\rrbracket)$.  By Corollary \ref{prop:largest-comm-quot} and Proposition \ref{prop: ideal-comparison} (see also Remark \ref{remark:nonss}) this quotient is a polynomial algebra in $[F:\bbQ_p](|\Phi_1| + d_Z) = [F:\bbQ_p](|\Delta| + 1 + d_Z)$ variables, where $d_Z := \dim_F(\mathbf{Z})$.  More precisely, the assumption $p > h + 1$ guarantees that $T_1 \cong Z_1 \times \langle \alpha^\vee(1 + \fm_F)\rangle_{\alpha \in \Phi}$ (see Lemma \ref{lem:boundconsq's}), and the assumption that $\pi$ possesses a central character implies $Z_1$ acts trivially on $\pi$.  Hence, if $z\in Z_1$, the element $z - 1 \in \gr_{\fm}(\bbF\llbracket I\rrbracket)$ acts by zero on $\gr_{\fm}(\pi^\vee)$, and by Lemma \ref{lem:GKdim} we obtain
	$$\dim_G(\pi)  = \dim\left(\textnormal{Supp}_{\textnormal{Spec}(\gr_{\fm}(\bbF\llbracket I\rrbracket)/\langle x_1, x_2, \ldots, x_{m}\rangle)}(\gr_{\fm}(\pi^\vee))\right)  \leq  [F:\bbQ_p](|\Delta| + 1).$$
\end{proof}

We would like to use the above lemmas in order to understand the Gelfand--Kirillov dimension of representations coming from global settings.  However, for general root systems, the bound of Lemma \ref{lem:GKbound} is in some sense incompatible with expectations coming from completed cohomology.  We explain this incompatibility below.

Suppose for simplicity that $\mathbf{G}$ is semisimple, and let $\bbG$ denote a connected reductive group over $\bbQ$ which is compact at infinity, and which satisfies $\bbG(\bbQ_p) \cong G$.  For a compact open subgroup $K = \prod_{\ell < \infty} K_{\ell} \subset \prod_{\ell < \infty} \bbG(\bbQ_\ell)$, we then consider the (finite-dimensional) space of functions
$$S_{\bbG}(K,\bbF) := \{f:\bbG(\bbQ)\backslash \bbG(\bbA_{\bbQ})/ \bbG(\bbR)^\circ  K \longrightarrow \bbF\}$$
(compare \cite[\S 4, Eqn. (4.2)]{gross:algmodforms}) and set
$$S_{\bbG}(K^p,\bbF) := \varinjlim_{K_p \subset \bbG(\bbQ_p)}S_{\bbG}(K^pK_p,\bbF) = \{f:\bbG(\bbQ)\backslash \bbG(\bbA_{\bbQ})/ \bbG(\bbR)^\circ  K^p \longrightarrow \bbF~\textnormal{locally constant}\}.$$
Here, $K^p = \prod_{p \neq \ell < \infty}K_\ell$, and the direct limit runs over the compact open subgroups of $\bbG(\bbQ_p)$.  The latter space is equipped with an action of $\bbG(\bbQ_p) \cong G$, and gives a smooth admissible representation.

Suppose now that $\pi$ is a smooth representation of $G$ which arises as a subquotient of $S_{\bbG}(K^p,\bbF)$ for some $\bbG$ and $K^p$ as above.  It is not unreasonable to guess that such a $\pi$ should satisfy $\dim_G(\pi) = [F:\bbQ_p]\dim(\sB) = [F:\bbQ_p]\cdot |\Phi^-|$, where $\sB$ denotes the flag variety of $\mathbf{G}$.  (See, for example, \cite[\S 3.1.1]{emerton:icm}.)  However, under the assumption that $\gr_{\fm}(\pi^\vee)$ satisfies the conditions of Lemma \ref{lem:GKdim} and $\Phi$ has rank greater than 1 and is not of type $\mathsf{A}_2$,  Lemma \ref{lem:GKbound} produces the bound 
$$\dim_G(\pi) \leq [F:\bbQ_p](|\Delta| + 1) < [F:\bbQ_p]\cdot |\Phi^-|,$$
which conflicts with these global expectations.  It is therefore unlikely that the strategy of Lemma \ref{lem:GKdim} can be used to confirm the expected bounds on $\dim_G(\pi)$.  In particular, in order to obtain the expected bounds, one is forced to work with non-commutative quotients of $\gr_{\mathfrak{m}}(\mathbb{F}\llbracket I\rrbracket)$.

We summarize this discussion as follows:

\begin{cor}
  \label{cor:GKcor}
  Let $\pi$ be a smooth admissible representation of $G$ admitting a central character.  Suppose $\Phi$ has rank greater than 1 and is not of type $\mathsf{A}_2$.  If $\dim_G(\pi) = [F:\mathbb{Q}_p]\cdot |\Phi^-|$ (as expected for representations coming from global settings), then $\textnormal{gr}_{\mathfrak{m}}(\pi^\vee)$ does not satisfy the conditions of Lemma \ref{lem:GKdim}.  In particular, the action of $\textnormal{gr}_{\mathfrak{m}}(\mathbb{F}\llbracket I\rrbracket)$ on $\textnormal{gr}_{\mathfrak{m}}(\pi^\vee)$ does not factor through the maximal commutative quotient of $\textnormal{gr}_{\mathfrak{m}}(\mathbb{F}\llbracket I\rrbracket)$.  
\end{cor}

\bibliographystyle{amsalpha}
\bibliography{refs}

@article {BT2,
    AUTHOR = {Bruhat, F. and Tits, J.},
     TITLE = {Groupes r\'{e}ductifs sur un corps local. {II}. {S}ch\'{e}mas en
              groupes. {E}xistence d'une donn\'{e}e radicielle valu\'{e}e},
   JOURNAL = {Inst. Hautes \'{E}tudes Sci. Publ. Math.},
  FJOURNAL = {Institut des Hautes \'{E}tudes Scientifiques. Publications
              Math\'{e}matiques},
    NUMBER = {60},
      YEAR = {1984},
     PAGES = {197--376},
      ISSN = {0073-8301},
   MRCLASS = {20G25 (14L15)},
  MRNUMBER = {756316},
MRREVIEWER = {James E. Humphreys},
       URL = {http://www.numdam.org/item?id=PMIHES_1984__60__5_0},
}

@book {bourbaki:lie1,
    AUTHOR = {Bourbaki, Nicolas},
     TITLE = {\'{E}l\'{e}ments de math\'{e}matique. {F}asc. {XXVI}.
              {G}roupes et alg\`ebres de {L}ie. {C}hapitre {I}: {A}lg\`ebres
              de {L}ie},
    SERIES = {Actualit\'{e}s Scientifiques et Industrielles [Current
              Scientific and Industrial Topics], No. 1285},
      NOTE = {Seconde \'{e}dition},
 PUBLISHER = {Hermann, Paris},
      YEAR = {1971},
     PAGES = {146 pp. (1 foldout)},
   MRCLASS = {22.80 (17.00)},
  MRNUMBER = {271276},
}

@book {bourbaki:lie2-3,
    AUTHOR = {Bourbaki, Nicolas},
     TITLE = {\'{E}l\'{e}ments de math\'{e}matique. {F}asc. {XXXVII}.
              {G}roupes et alg\`ebres de {L}ie. {C}hapitre {II}:
              {A}lg\`ebres de {L}ie libres. {C}hapitre {III}: {G}roupes de
              {L}ie},
    SERIES = {Actualit\'{e}s Scientifiques et Industrielles [Current
              Scientific and Industrial Topics], No. 1349},
 PUBLISHER = {Hermann, Paris},
      YEAR = {1972},
     PAGES = {320},
   MRCLASS = {17-01 (22-01)},
  MRNUMBER = {573068},
MRREVIEWER = {H.\ Reitberger},
}

@book {bourbaki:lie,
    AUTHOR = {Bourbaki, Nicolas},
     TITLE = {\'{E}l\'{e}ments de math\'{e}matique},
      NOTE = {Groupes et alg\`ebres de Lie. Chapitres 4, 5 et 6. [Lie groups
              and Lie algebras. Chapters 4, 5 and 6]},
 PUBLISHER = {Masson, Paris},
      YEAR = {1981},
     PAGES = {290},
      ISBN = {2-225-76076-4},
   MRCLASS = {17-02 (00A05)},
  MRNUMBER = {647314},
}

@article {BP,
    AUTHOR = {Breuil, Christophe and Pa\v{s}k\={u}nas, Vytautas},
     TITLE = {Towards a modulo {$p$} {L}anglands correspondence for {${\rm
              GL}_2$}},
   JOURNAL = {Mem. Amer. Math. Soc.},
  FJOURNAL = {Memoirs of the American Mathematical Society},
    VOLUME = {216},
      YEAR = {2012},
    NUMBER = {1016},
     PAGES = {vi+114},
      ISSN = {0065-9266},
      ISBN = {978-0-8218-5227-9},
   MRCLASS = {11F70 (11F80 22E50)},
  MRNUMBER = {2931521},
MRREVIEWER = {Gabor Wiese},
       DOI = {10.1090/S0065-9266-2011-00623-4},
       URL = {https://doi.org/10.1090/S0065-9266-2011-00623-4},
}

@incollection {berger:bourbaki,
    AUTHOR = {Berger, Laurent},
     TITLE = {La correspondance de {L}anglands locale {$p$}-adique pour
              {${\rm GL}_2({\bf Q}_p)$}},
      NOTE = {S\'{e}minaire Bourbaki. Vol. 2009/2010. Expos\'{e}s 1012--1026},
   JOURNAL = {Ast\'{e}risque},
  FJOURNAL = {Ast\'{e}risque},
    NUMBER = {339},
      YEAR = {2011},
     PAGES = {Exp. No. 1017, viii, 157--180},
      ISSN = {0303-1179},
      ISBN = {978-2-85629-326-3},
   MRCLASS = {11S37 (11F70 11F85 22E50)},
  MRNUMBER = {2906353},
MRREVIEWER = {B. Sury},
}

@incollection {breuil:bourbaki,
    AUTHOR = {Breuil, Christophe},
     TITLE = {Correspondance de {L}anglands {$p$}-adique, compatibilit\'{e}
              local-global et applications [d'apr\`es {C}olmez, {E}merton,
              {K}isin, {$\ldots$}]},
      NOTE = {S\'{e}minaire Bourbaki: Vol. 2010/2011. Expos\'{e}s 1027--1042},
   JOURNAL = {Ast\'{e}risque},
  FJOURNAL = {Ast\'{e}risque},
    NUMBER = {348},
      YEAR = {2012},
     PAGES = {Exp. No. 1031, viii, 119--147},
      ISSN = {0303-1179},
      ISBN = {978-2-85629-351-5},
   MRCLASS = {11S37 (11F70 22E50)},
  MRNUMBER = {3050714},
MRREVIEWER = {Volker J. Heiermann},
}

@inproceedings {breuil:ICM,
    AUTHOR = {Breuil, Christophe},
     TITLE = {The emerging {$p$}-adic {L}anglands programme},
 BOOKTITLE = {Proceedings of the {I}nternational {C}ongress of
              {M}athematicians. {V}olume {II}},
     PAGES = {203--230},
 PUBLISHER = {Hindustan Book Agency, New Delhi},
      YEAR = {2010},
   MRCLASS = {22E50 (11F70 11F80)},
  MRNUMBER = {2827792},
MRREVIEWER = {Benjamin Schraen},
}

@article {BHHMS,
    AUTHOR = {Breuil, Christophe and Herzig, Florian and Hu, Yongquan and
              Morra, Stefano and Schraen, Benjamin},
     TITLE = {Gelfand-{K}irillov dimension and {${\rm mod}\, p$} cohomology
              for {$\rm GL_2$}},
   JOURNAL = {Invent. Math.},
  FJOURNAL = {Inventiones Mathematicae},
    VOLUME = {234},
      YEAR = {2023},
    NUMBER = {1},
     PAGES = {1--128},
      ISSN = {0020-9910},
   MRCLASS = {11F80 (11R52)},
  MRNUMBER = {4635831},
MRREVIEWER = {Yiwen Ding},
       DOI = {10.1007/s00222-023-01202-8},
       URL = {https://doi.org/10.1007/s00222-023-01202-8},
}

@incollection {conrad:nonsplit,
    AUTHOR = {Conrad, Brian},
     TITLE = {Non-split reductive groups over {${\bf Z}$}},
 BOOKTITLE = {Autours des sch\'{e}mas en groupes. {V}ol. {II}},
    SERIES = {Panor. Synth\`eses},
    VOLUME = {46},
     PAGES = {193--253},
 PUBLISHER = {Soc. Math. France, Paris},
      YEAR = {2015},
   MRCLASS = {14L15 (14G15 14G35 14L35 20G30)},
  MRNUMBER = {3525597},
MRREVIEWER = {Boris \`E. Kunyavski\u{\i}},
       DOI = {10.1017/CBO9781316092439},
       URL = {https://doi.org/10.1017/CBO9781316092439},
}

@incollection {conrad:reductive,
    AUTHOR = {Conrad, Brian},
     TITLE = {Reductive group schemes},
 BOOKTITLE = {Autour des sch\'{e}mas en groupes. {V}ol. {I}},
    SERIES = {Panor. Synth\`eses},
    VOLUME = {42/43},
     PAGES = {93--444},
 PUBLISHER = {Soc. Math. France, Paris},
      YEAR = {2014},
   MRCLASS = {14L15},
  MRNUMBER = {3362641},
}

@article {DottoLe,
    AUTHOR = {Dotto, Andrea and Le, Daniel},
     TITLE = {Diagrams in the {$\textrm{mod}~ p$} cohomology of {S}himura
              curves},
   JOURNAL = {Compos. Math.},
  FJOURNAL = {Compositio Mathematica},
    VOLUME = {157},
      YEAR = {2021},
    NUMBER = {8},
     PAGES = {1653--1723},
      ISSN = {0010-437X},
   MRCLASS = {11F80 (11S37 22E50)},
  MRNUMBER = {4283560},
MRREVIEWER = {Neil P. Dummigan},
       DOI = {10.1112/s0010437x21007375},
       URL = {https://doi.org/10.1112/s0010437x21007375},
}

@misc{DottoLeHung,
      title={Cohomology of $p$-adic {C}hevalley groups}, 
      author={Andrea Dotto and Bao V. Le Hung},
      year={2025},
      eprint={2507.13500},
      archivePrefix={arXiv},
      primaryClass={math.NT},
      url={https://arxiv.org/abs/2507.13500}, 
      note={\url{https://arxiv.org/abs/2507.13500}},
}

@inproceedings {emerton:icm,
    AUTHOR = {Emerton, Matthew},
     TITLE = {Completed cohomology and the {$p$}-adic {L}anglands program},
 BOOKTITLE = {Proceedings of the {I}nternational {C}ongress of
              {M}athematicians---{S}eoul 2014. {V}ol. {II}},
     PAGES = {319--342},
 PUBLISHER = {Kyung Moon Sa, Seoul},
      YEAR = {2014},
   MRCLASS = {11F70 (22D12)},
  MRNUMBER = {3728617},
MRREVIEWER = {Ivan Mati\'{c}},
}

@article {EGS,
    AUTHOR = {Emerton, Matthew and Gee, Toby and Savitt, David},
     TITLE = {Lattices in the cohomology of {S}himura curves},
   JOURNAL = {Invent. Math.},
  FJOURNAL = {Inventiones Mathematicae},
    VOLUME = {200},
      YEAR = {2015},
    NUMBER = {1},
     PAGES = {1--96},
      ISSN = {0020-9910},
   MRCLASS = {11F80 (14G35 22E50)},
  MRNUMBER = {3323575},
MRREVIEWER = {Peter Bruin},
       DOI = {10.1007/s00222-014-0517-0},
       URL = {https://doi.org/10.1007/s00222-014-0517-0},
}

@article {gross:algmodforms,
    AUTHOR = {Gross, Benedict H.},
     TITLE = {Algebraic modular forms},
   JOURNAL = {Israel J. Math.},
  FJOURNAL = {Israel Journal of Mathematics},
    VOLUME = {113},
      YEAR = {1999},
     PAGES = {61--93},
      ISSN = {0021-2172},
   MRCLASS = {11F55 (11F80 20G30 22E55)},
  MRNUMBER = {1729443},
MRREVIEWER = {Stefan K\"{u}hnlein},
       DOI = {10.1007/BF02780173},
       URL = {https://doi.org/10.1007/BF02780173},
}

@article {GLS,
    AUTHOR = {Ghate, Eknath and Le, Daniel and Sheth, Mihir},
     TITLE = {Non-admissible irreducible representations of {$p$}-adic
              {${\rm GL}_n$} in characteristic {$p$}},
   JOURNAL = {Represent. Theory},
  FJOURNAL = {Representation Theory. An Electronic Journal of the American
              Mathematical Society},
    VOLUME = {27},
      YEAR = {2023},
     PAGES = {1088--1101},
   MRCLASS = {22E50 (11S37)},
  MRNUMBER = {4664337},
MRREVIEWER = {Peter Dillery},
       DOI = {10.1090/ert/660},
       URL = {https://doi.org/10.1090/ert/660},
}

@article {HuWang,
    AUTHOR = {Hu, Yongquan and Wang, Haoran},
     TITLE = {Multiplicity one for the {$\textrm{mod}\, p$} cohomology of
              {S}himura curves: the tame case},
   JOURNAL = {Math. Res. Lett.},
  FJOURNAL = {Mathematical Research Letters},
    VOLUME = {25},
      YEAR = {2018},
    NUMBER = {3},
     PAGES = {843--873},
      ISSN = {1073-2780},
   MRCLASS = {11F70},
  MRNUMBER = {3847337},
MRREVIEWER = {Zhengyu Mao},
       DOI = {10.4310/MRL.2018.v25.n3.a6},
       URL = {https://doi.org/10.4310/MRL.2018.v25.n3.a6},
}

@article {HuWang2,
    AUTHOR = {Hu, Yongquan and Wang, Haoran},
     TITLE = {On the {$\bmod p$} cohomology for {$\rm GL_2$}: the
              non-semisimple case},
   JOURNAL = {Camb. J. Math.},
  FJOURNAL = {Cambridge Journal of Mathematics},
    VOLUME = {10},
      YEAR = {2022},
    NUMBER = {2},
     PAGES = {261--431},
      ISSN = {2168-0930},
   MRCLASS = {11F70 (22E50)},
  MRNUMBER = {4461834},
MRREVIEWER = {Yiwen Ding},
       DOI = {10.4310/cjm.2022.v10.n2.a1},
       URL = {https://doi.org/10.4310/cjm.2022.v10.n2.a1},
}

@article {lahirisorensen,
    AUTHOR = {Lahiri, Aranya and Sorensen, Claus},
     TITLE = {Rigid vectors in {$p$}-adic principal series representations},
   JOURNAL = {Israel J. Math.},
  FJOURNAL = {Israel Journal of Mathematics},
    VOLUME = {259},
      YEAR = {2024},
    NUMBER = {1},
     PAGES = {427--459},
      ISSN = {0021-2172},
   MRCLASS = {22E50},
  MRNUMBER = {4732375},
       DOI = {10.1007/s11856-023-2495-7},
       URL = {https://doi.org/10.1007/s11856-023-2495-7},
}

@article {lazard,
    AUTHOR = {Lazard, Michel},
     TITLE = {Groupes analytiques {$p$}-adiques},
   JOURNAL = {Inst. Hautes \'{E}tudes Sci. Publ. Math.},
  FJOURNAL = {Institut des Hautes \'{E}tudes Scientifiques. Publications
              Math\'{e}matiques},
    NUMBER = {26},
      YEAR = {1965},
     PAGES = {389--603},
      ISSN = {0073-8301},
   MRCLASS = {14.50},
  MRNUMBER = {209286},
       URL = {http://www.numdam.org/item?id=PMIHES_1965__26__389_0},
}

@article {Le:wild,
    AUTHOR = {Le, Daniel},
     TITLE = {Multiplicity one for wildly ramified representations},
   JOURNAL = {Algebra Number Theory},
  FJOURNAL = {Algebra \& Number Theory},
    VOLUME = {13},
      YEAR = {2019},
    NUMBER = {8},
     PAGES = {1807--1827},
      ISSN = {1937-0652},
   MRCLASS = {11R39 (11F80)},
  MRNUMBER = {4017535},
MRREVIEWER = {Atsushi Yamagami},
       DOI = {10.2140/ant.2019.13.1807},
       URL = {https://doi.org/10.2140/ant.2019.13.1807},
}

@article {LMS,
    AUTHOR = {Le, Daniel and Morra, Stefano and Schraen, Benjamin},
     TITLE = {Multiplicity one at full congruence level},
   JOURNAL = {J. Inst. Math. Jussieu},
  FJOURNAL = {Journal of the Institute of Mathematics of Jussieu. JIMJ.
              Journal de l'Institut de Math\'{e}matiques de Jussieu},
    VOLUME = {21},
      YEAR = {2022},
    NUMBER = {2},
     PAGES = {637--658},
      ISSN = {1474-7480},
   MRCLASS = {11F33 (11F80 20C33)},
  MRNUMBER = {4386824},
MRREVIEWER = {Karam Deo Shankhadhar},
       DOI = {10.1017/S1474748020000225},
       URL = {https://doi.org/10.1017/S1474748020000225},
}

@MISC {root-lemma-mo,
    TITLE = {Why is the root poset is graded by height?},
    AUTHOR = {McNamara, Peter},
    HOWPUBLISHED = {MathOverflow},
    NOTE = {Answer URL:\url{https://mathoverflow.net/q/283350} (version: 2017-10-12), User URL:\url{https://mathoverflow.net/users/425/peter-mcnamara}},
    URL = {https://mathoverflow.net/q/283350}
}

@book {platonovrapinchuk,
    AUTHOR = {Platonov, Vladimir and Rapinchuk, Andrei},
     TITLE = {Algebraic groups and number theory},
    SERIES = {Pure and Applied Mathematics},
    VOLUME = {139},
      NOTE = {Translated from the 1991 Russian original by Rachel Rowen},
 PUBLISHER = {Academic Press, Inc., Boston, MA},
      YEAR = {1994},
     PAGES = {xii+614},
      ISBN = {0-12-558180-7},
   MRCLASS = {11E57 (11-02 20Gxx)},
  MRNUMBER = {1278263},
}

@article {schneiderstuhler,
    AUTHOR = {Schneider, Peter and Stuhler, Ulrich},
     TITLE = {Representation theory and sheaves on the {B}ruhat-{T}its
              building},
   JOURNAL = {Inst. Hautes \'{E}tudes Sci. Publ. Math.},
  FJOURNAL = {Institut des Hautes \'{E}tudes Scientifiques. Publications
              Math\'{e}matiques},
    NUMBER = {85},
      YEAR = {1997},
     PAGES = {97--191},
      ISSN = {0073-8301},
   MRCLASS = {22E50 (11F70 20G25)},
  MRNUMBER = {1471867},
MRREVIEWER = {Ernst-Wilhelm Zink},
       URL = {http://www.numdam.org/item?id=PMIHES_1997__85__97_0},
}

@book {schneider:padicliegroups,
    AUTHOR = {Schneider, Peter},
     TITLE = {{$p$}-adic {L}ie groups},
    SERIES = {Grundlehren der mathematischen Wissenschaften [Fundamental
              Principles of Mathematical Sciences]},
    VOLUME = {344},
 PUBLISHER = {Springer, Heidelberg},
      YEAR = {2011},
     PAGES = {xii+254},
      ISBN = {978-3-642-21146-1},
   MRCLASS = {22E20 (16S34 22E35)},
  MRNUMBER = {2810332},
MRREVIEWER = {Dubravka Ban},
       DOI = {10.1007/978-3-642-21147-8},
       URL = {https://doi.org/10.1007/978-3-642-21147-8},
}

@article {schraen:finpres,
    AUTHOR = {Schraen, Benjamin},
     TITLE = {Sur la pr\'{e}sentation des repr\'{e}sentations supersinguli\`eres de
              {${\rm GL}_2(F)$}},
   JOURNAL = {J. Reine Angew. Math.},
  FJOURNAL = {Journal f\"{u}r die Reine und Angewandte Mathematik. [Crelle's
              Journal]},
    VOLUME = {704},
      YEAR = {2015},
     PAGES = {187--208},
      ISSN = {0075-4102},
   MRCLASS = {22E50 (11F33 11S37)},
  MRNUMBER = {3365778},
MRREVIEWER = {Gabor Wiese},
       DOI = {10.1515/crelle-2013-0049},
       URL = {https://doi.org/10.1515/crelle-2013-0049},
}

@book {SGA3,
    AUTHOR = {Artin, M. and Bertin, J. E. and Demazure, M. and Gabriel, P.
              and Grothendieck, A. and Raynaud, M. and Serre, J.-P.},
     TITLE = {Sch\'{e}mas en groupes. {F}asc. 7: {E}xpos\'{e}s 23 \`a 26},
      NOTE = {Premi\`ere \'{e}dition,
              S\'{e}minaire de G\'{e}om\'{e}trie Alg\'{e}brique de l'Institut des Hautes
              \'{E}tudes Scientifiques, 1963/64, dirig\'{e} par Michel Demazure et
              Alexander Grothendieck},
 PUBLISHER = {Institut des Hautes \'{E}tudes Scientifiques, Paris},
      YEAR = {1965},
     PAGES = {ii+260 pp. (not consecutively paged)},
   MRCLASS = {14.50},
  MRNUMBER = {207710},
}

@article {totaro,
    AUTHOR = {Totaro, Burt},
     TITLE = {Euler characteristics for {$p$}-adic {L}ie groups},
   JOURNAL = {Inst. Hautes \'{E}tudes Sci. Publ. Math.},
  FJOURNAL = {Institut des Hautes \'{E}tudes Scientifiques. Publications
              Math\'{e}matiques},
    NUMBER = {90},
      YEAR = {1999},
     PAGES = {169--225 (2001)},
      ISSN = {0073-8301},
   MRCLASS = {22E50 (11S85 17B45 22E35)},
  MRNUMBER = {1813226},
MRREVIEWER = {Nguy\cftil{e}n Qu\^{o}c Th\'{a}ng},
       URL = {http://www.numdam.org/item?id=PMIHES_1999__90__169_0},
}

@article {wu:finpres,
    AUTHOR = {Wu, Zhixiang},
     TITLE = {A note on presentations of supersingular representations of
              {$\textrm{GL}_2(F)$}},
   JOURNAL = {Manuscripta Math.},
  FJOURNAL = {Manuscripta Mathematica},
    VOLUME = {165},
      YEAR = {2021},
    NUMBER = {3-4},
     PAGES = {583--596},
      ISSN = {0025-2611},
   MRCLASS = {22E50 (11F70)},
  MRNUMBER = {4280498},
MRREVIEWER = {Zhe Chen},
       DOI = {10.1007/s00229-020-01224-z},
       URL = {https://doi.org/10.1007/s00229-020-01224-z},
}

\end{document}